\pdfoutput=1

\documentclass[10pt,reqno]{amsart}
\usepackage{amsfonts,amssymb}
\usepackage{color}
\usepackage{xcolor}
\usepackage[mathscr]{eucal}
\usepackage{amsmath}
\usepackage{amsthm}
\usepackage{mathrsfs}
\usepackage{amsbsy}
\usepackage{dsfont}
\usepackage{mathtools}
\usepackage{bbm}
\usepackage{subcaption}
\usepackage{wasysym}
\usepackage{multirow}
\usepackage{graphicx,epstopdf}
\usepackage[pdftex,bookmarks,colorlinks,breaklinks]{hyperref}  
\usepackage{color}
\usepackage[ruled,vlined]{algorithm2e}

\definecolor{dullmagenta}{rgb}{0.4,0,0.4}   
\definecolor{darkblue}{rgb}{0,0,0.4}
\definecolor{orange}{rgb}{1,0.5,0}
\hypersetup{linkcolor=red,citecolor=blue,filecolor=dullmagenta,urlcolor=darkblue} 
\graphicspath{{C:\Users\pmxbk1\Documents\Mathematics\NEW STUFF\MaxCut Paper}}

\DeclarePairedDelimiter\floor{\lfloor}{\rfloor}
\DeclareGraphicsExtensions{.png,.pdf,.jpg}
\usepackage{stmaryrd}
\input xypic
\xyoption{all}

\theoremstyle{plain}
\newtheorem{theorem}{Theorem}[section]
\newtheorem{lemma}[theorem]{Lemma}
\newtheorem{corollary}[theorem]{Corollary}
\newtheorem{proposition}[theorem]{Proposition}

\newtheorem{definition}[theorem]{Definition}

\newcommand{\mc}[1]{\mathrm{mc}(#1)}
\newcommand{\TV}{\mathrm{TV}}
\newcommand{\argmax}{\mathrm{argmax}}
\newcommand{\argmin}{\mathrm{argmin}}
\let\e\varepsilon
\newcommand{\vz}[1]{\ensuremath{\mathbb{#1}}}
\newcommand{\R}{{\vz R}}
\newcommand{\N}{{\vz N}}

\newcounter{hdps}
\newenvironment{asm}[1]
  {\refstepcounter{hdps}
    \begin{algorithm}\renewcommand\thealgocf{#1}}
  {\end{algorithm}\addtocounter{algocf}{-1}}


\begin{document}
\baselineskip = 5mm


\title[Max-Cut Approximation: Graph Based MBO Scheme]{A Max-Cut Approximation Using A Graph Based MBO Scheme}

\maketitle

\begin{center}
\large{Blaine Keetch, Yves van Gennip},\\ 
\normalsize{School of Mathematical Sciences,\\ The University of Nottingham,\\ University Park, Nottingham, NG7 2RD.}
\end{center}

\begin{abstract}

The Max-Cut problem is a well known combinatorial optimization problem. In this paper we describe a fast approximation method. Given a graph $G$, we want to find a cut whose size is maximal among all possible cuts. A cut is a partition of the vertex set of $G$ into two disjoint subsets. For an unweighted graph, the size of the cut is the number of edges that have one vertex on either side of the partition; we also consider a weighted version of the problem where each edge contributes a nonnegative weight to the cut.

We introduce the signless Ginzburg-Landau functional and prove that this functional $\Gamma$-converges to a Max-Cut objective functional. We approximately minimize this functional using a graph based signless Merriman-Bence-Osher scheme, which uses a signless Laplacian. We show experimentally that on some classes of graphs the resulting algorithm produces more accurate maximum cut approximations than the current state-of-the-art approximation algorithm. One of our methods of minimizing the functional results in an algorithm with a time complexity of $\mathcal{O}(|E|)$, where $|E|$ is the total number of edges on $G$.

\end{abstract}


\section{Introduction}

\subsection{Maximum cut}\label{sec:maximumcut}

Given an undirected (edge-)weighted graph $G = (V,E,\omega)$, a cut $V_{-1}|V_1$ is a partition of the node set $V$ into two disjoint subsets $V_{-1}$ and $V_1$. The size of a cut $C=V_{-1}|V_1$, denoted by $s(C)$, is the sum of all the weights corresponding to edges that have one end vertex in $V_{-1}$ and one in $V_1$. The maximum cut (Max-Cut) problem is the problem of finding a cut $C^*$ such that for all cuts $C$, $s(C) \leq s(C^*)$. We call such a $C^*$ a maximum cut and say $\mc{G}:=s(C^*)$ is the maximum cut value of the graph $G$. The Max-Cut problem for an unweighted graph is a special case of the Max-Cut problem on a weighted graph which we obtain by assuming all edge weights are $1$. Finding an unweighted graph's Max-Cut is equivalent to finding a bipartite subgraph with the largest number of edges possible. In fact, for an unweighted bipartite graph $\mc{G}=|E|$. 

The Max-Cut problem is an NP-hard problem; assuming P $\neq$ NP no solution can be acquired in polynomial time. There are a variety of polynomial time approximation algorithms for this problem \cite{goemans1995,bylka1999,trevisan2012}. Some Max-Cut approximation algorithms have a proven lower bound on their accuracy, which asserts the existence of a $\beta\in [0,1]$ such that, for all output cuts $C$ obtained by the algorithm, $s(C) \geq \beta \mc{G}$. We call such a $\beta$ a performance guarantee. For algorithms that incorporate stochastic steps, such a lower bound typically takes the form $E[s(C)] \geq \beta \mc{G}$ instead, where $E[s(C)]$ denotes the expected value of the size of the output cut.

In recent years a new type of approach to approximating such graph problems has gained traction. Models from the world of partial differential equations and variational methods that exhibit behaviour of the kind that could be helpful in solving the graph problem are transcribed from their usual continuum formulation to a graph based model. The resulting discrete model can then be solved using techniques from numerical analysis and scientific computing. Examples of problems that have successfully been tackled in this manner include data classification \cite{bertozzi2012}, image segmentation \cite{calatroni2016}, and community detection \cite{hu2013}. In this paper we use a variation on the graph Ginzburg-Landau functional, which was introduced in \cite{bertozzi2012}, to construct an algorithm which approximately solves the Max-Cut problem on simple undirected weighted graphs.

We compare our method with the Goemans-Williamson (GW) algorithm \cite{goemans1995}, which is the current state-of-the-art method for approximately solving the Max-Cut problem. In \cite{goemans1995} the authors solve a relaxed Max-Cut objective function and intersect the solution with a random hyperplane in a $n$-dimensional sphere. It is proven that if gw($C$) is the size of the cut produced by the Goemans-Williamson algorithm, then its expected value $E$[gw$(C)]$ satisfies the inequality $E$[gw$(C)] \geq \beta \mc{G}$ where $\beta = 0.878$ (rounded down). If the Unique Games Conjecture \cite{khot2002} is true, the GW algorithm has the best performance guarantee that is possible for a polynomial time approximation algorithm \cite{khot2007}. It has been proven that approximately solving the Max-Cut problem with a performance guarantee of $\frac{16}{17}\approx 0.941$ or better is NP-hard \cite{trevisan2000}.

Finding $\mc{G}$ is equivalent to finding the ground state of the Ising Hamiltonian in Ising spin models \cite{haribara2016, barahona1988}, and 0/1 linear programming problems can be restated as Max-Cut problems \cite{lasserre2016}.

\subsection{Signless Ginzburg-Landau functional}\label{sec:GL}

Spectral graph theory \cite{chung1997} explores the relationships between the spectra of graph operators, such as graph Laplacians (see Section~\ref{sec:setup}), and properties of graphs. For example, the multiplicity of the zero eigenvalue of the (unnormalised, random walk, or symmetrically normalised) graph Laplacian is equal to the number of connected components of the graph. Such properties lie at the basis of the successful usage of the graph Laplacian in graph clustering, such as in spectral clustering \cite{von2007} and in clustering and classification methods that use the graph Ginzburg-Landau functional \cite{bertozzi2012}
\begin{equation}\label{eq:graphGL}
f_\e(u) := \frac12 \sum_{i,j\in V} \omega_{ij} (u_i-u_j)^2 + \frac1\e \sum_{i\in V} W(u_i).
\end{equation}
Here $u: V\to \R$ is a real-valued function defined on the node set $V$, with value $u_i$ on node $i$, $\omega_{ij}$ is a positive weight associated with the edge between nodes $i$ and $j$ (and $\omega_{ij}=0$ if such an edge is absent), and $W(x):= (x^2-1)^2$ is a double-well potential with minima at $x=\pm 1$. In Section~\ref{sec:setup} we will introduce our setting and notation more precisely.

The method we use in this paper is based on a variation of $f_\e$, we call the \textit{signless Ginzburg-Landau functional}:
\begin{equation}\label{eq:signlessgraphGL}
f_\e^+(u): = \frac12 \sum_{i,j\in V} \omega_{ij} (u_i+u_j)^2 + \frac1\e \sum_{i\in V} W(u_i).
\end{equation}
This nomenclature is suggested by the fact that the signless graph Laplacians are related to $f_\e^+$ in a similar way as the graph Laplacians are related to $f_\e$, as we will see in Section~\ref{sec:setup}. Signless graph Laplacians have been studied because of the connections between their spectra and bipartite subgraphs \cite{desai1994}. In \cite{hein2007, van2012} the authors derive a graph difference operator and a graph divergence operator to form a graph Laplacian operator. In this paper we mimic this framework by deriving a signless difference operator and a signless divergence operator to form a signless Laplacian operator. Whereas the graph Laplacian operator is a discretization of the continuum Laplacian operator, the continuum analogue of the signless Laplacian is an averaging operator which is the subject of current and future research.

The functional $f_\e$ is useful in clustering and classification problems, because minimizers of $f_\e$ (in the presence of some constraint or additional term, to prevent trivial minimizers) will be approximately binary (with values close to $\pm 1$), because of the double-well potential term, and will have similar values on nodes that are connected by highly weighted edges, because of the first term in $f_\e$. This intuition can be formalised using the language of $\Gamma$-convergence \cite{dalmaso1993}. In analogy with the continuum case in \cite{modica1977,modica1987}, it  was proven in \cite{van2012} that if $\e\downarrow 0$, then $f_\e$ $\Gamma$-converges to
\[
f_0(u) := \begin{cases}
2\TV(u), &\text{if } u \text{ only takes the values } \pm1,  \\
\infty, &\text{otherwise},
\end{cases}
\]
where TV$(u):= \frac12 \sum_{i,j\in V} \omega_{ij} |u_i-u_j|$ is the graph total variation\footnote{The multiplicative factor $2$ in $f_0$ above differs from that in \cite{van2012} because in the current paper we choose different locations for the wells of $W$.}. Together with an equi-coercivity property, which we will return to in more detail in Section~\ref{sec:fe+}, this $\Gamma$-convergence result guarantees that minimizers of $f_\e$ converge to minimizers of $f_0$ as $\e \downarrow 0$. If $u$ only takes the values $\pm 1$, we note that $\TV(u) = 2s(C)$, where $C=V_{-1}|V_1$ is the cut given by $V_{\pm 1}:= \{i\in V: u_i = \pm 1\}$. Hence minimizers $u^\e$ of $f_\e$ are expected to approximately solve the {\it minimal} cut problem, if we let $V_{\pm 1} \approx \{i\in V: u_i^\e \approx \pm 1\}$.

In Section~\ref{sec:fe+} we prove that $f_\e^+$ $\Gamma$-converges to a limit functional whose minimizers solve the Max-Cut problem. Hence, we expect minimizers $u^\e$ of $f_\e^+$ to approximately solve the Max-Cut problem, if we consider the cut $C=V_{-1}|V_1$, with $V_{\pm 1} = \{i\in V: u_i^\e \approx \pm 1\}$.

\subsection{Graph MBO scheme}

There are various ways in which the minimization of $f_\e^+$ can be attempted. One such way, which can be explored in a future publication, is to use a gradient flow method. In the case of $f_\e$ the gradient flow is given by an Allen-Cahn type equation on graphs \cite{bertozzi2012,van2014},
\begin{equation}\label{eq:graphAC}
\frac{du_i}{dt} = -(\Delta u)_i - \frac1\e d_i^{-r}W'(u_i),
\end{equation}
where $\Delta u$ is a graph Laplacian of $u$, $d_i$ the degree of node $i$, and $r\in [0,1]$ a parameter (see Section~\ref{sec:setup} for further details). This can be solved using a combination of convex splitting and spectral truncation. In the case of $f_\e^+$ such an approach would lead to a similar equation and scheme, with the main difference being the use of a signless graph Laplacian instead of a graph Laplacian.

In this paper, however, we have opted for an alternative approach, which is also inspired by similar approaches which have been developed for the $f_\e$ case. The continuum Merriman-Bence-Osher (MBO) scheme \cite{MBO1992,MBO1993} involves iteratively solving the diffusion equation over a small time step $\tau$ and thresholding the solution to an indicator function. For a short diffusion time $\tau$ this scheme approximates motion by mean curvature \cite{barles1995}.  This scheme has been adapted to a graph setting \cite{merkurjev2013,van2014}. Heuristically it is expected that the outcome of the graph MBO scheme closely approximates minimizers of $f_\e$, as the diffusion step involves solving  $\frac{du_i}{dt} = -(\Delta u)_i $ and the thresholding step has a similar effect as the nonlinearity $-\frac1\e W'(u_i)$ in \eqref{eq:graphAC}. Experimental results strengthen this expectation, however rigorous confirmation is still lacking. 

In order to approximately minimize $f_\e^+$, and consequently approximately solve the Max-Cut problem, we use an MBO type scheme in which we replace the graph Laplacian in the diffusion step by a signless graph Laplacian. We use two methods to compute this step: (1) a spectral method, adapted from the one in \cite{bertozzi2012}, which allows us to use a small subset of the eigenfunctions, which correspond to the smallest eigenvalues of the graph Laplacian, and (2) an Euler method.

The usefulness of (normalised) signless graph Laplacians when attempting to find maximum cuts can be intuitively understood from the fact that their spectra are in a sense (which is made precise in Proposition~\ref{prop:eigenpairs}) the reverse of the spectra of the corresponding (normalised) graph Laplacians. Hence, where a standard graph Laplacian driven diffusion leads to clustering patterns according to the eigenfunctions corresponding to its smallest eigenvalues, `diffusion' driven by a signless graph Laplacian leads to patterns resembling the eigenfunctions corresponding to the smallest eigenvalues of that signless graph Laplacian and thus the largest eigenvalues of the corresponding standard graph Laplacian. 

\subsection{Structure of the paper}

In Section~\ref{sec:Graphs} we explain the notation we use in this paper and give some preliminary results. Section~\ref{sec:MCGW} gives a precise formulation of the Max-Cut problem and discusses the Goemans-Williamson algorithm in more detail. In Section~\ref{sec:fe+} we introduce the signless graph Ginzburg-Landau functional $f_{\varepsilon}^+$ and use $\Gamma$-convergence techniques to prove that minimizers of $f_\e^+$ can be used to find approximate maximum cuts. We describe the signless MBO algorithm we use to find approximate minimizers of $f_\e^+$ in Section~\ref{sec:signlessMBO} and discuss the results we get in Section~\ref{sec:results}. We analyse the influence of our parameter choices in Section~\ref{sec:parameters} and conclude the paper in Section~\ref{sec:conclusions}.


\section{Setup and notation}\label{sec:Graphs}

\subsection{Graph based operators and functionals}\label{sec:setup}

In this paper we will consider non-empty finite, simple\footnote{By `simple' we mean 'without self-loops and without multiple edges between the same pair of vertices'. Note that removing self-loops from a graph does not change its maximum cut.}, undirected graphs $G = (V,E,\omega)$ without isolated nodes, with vertex set (or node set) $V$, edge set $E \subset V^2$ and non-negative edge weights $\omega$. We denote the set of all such graphs by $\mathcal{G}$.  By assumption $V$ has finite cardinality, which we denote by $n:=|V| \in \mathbb{N}$\footnote{For definiteness we use the convention $0\not\in\N$.}. We assume a node labelling such that $V=\{1, \ldots, n\}$. When $i,j\in V$ are nodes, the undirected edge between $i$ and $j$, if present, is denoted by $(i,j)$. The edge weight corresponding to this edge is $\omega_{\ij}>0$. Since $G$ is undirected, we identify $(i,j)$ with $(j,i)$ in $E$. Within this framework we can also consider unweighted graphs, which correspond to the cases in which, for all $(i,j)\in E$, $\omega_{ij}=1$.

We define $\mathcal{V}$ to be the set consisting of all node functions $u: V \to \R$ and $\mathcal{E}$ to be the set of edge functions $\varphi: E \rightarrow \R$. We will use the notation $u_i:=u(i)$ and $\varphi_{ij} := \varphi(i,j)$ for functions $u\in \mathcal{V}$ and $\varphi\in \mathcal{E}$, respectively. For notational convenience, we will typically associate $\varphi\in \mathcal{E}$ with its extension to $V^2$ obtained by setting $\varphi_{ij}=0$ if $(i,j)\not\in E$. We also extend $\omega$ to $V^2$ in this way: if $(i,j)\not\in E$, then $\omega_{ij}=0$. Because $G\in \mathcal{G}$ is undirected, we have for all $(i,j)\in E$, $\omega_{ij}=\omega_{ji}$. Because $G\in \mathcal{G}$ is simple, for all $i\in V$, $(i,i) \notin E$.  The degree of a node $i$ is $d_i := \sum_{j \in V} \omega_{ij}$. Because $G \in \mathcal{G}$ does not contain isolated nodes, we have for all $i \in V, d_i > 0$. 

As shown in \cite{hein2007}, it is possible for $\mathcal{V}$ and $\mathcal{E}$ to be defined for directed graphs, but we will not pursue these ideas here. 

To introduce the graph Laplacians and signless graph Laplacians we use and extend the structure that was used in \cite{hein2007,van2012,van2014}.
We define the inner products on $\mathcal{V}$ and $\mathcal{E}$ as
\[
\langle u,v \rangle_{\mathcal{V}} := \displaystyle\sum_{i \in V} u_iv_id_i^r, \qquad \langle \varphi,\phi \rangle_{\mathcal{E}} := \frac{1}{2}\displaystyle\sum_{i,j \in V} \varphi_{ij}\phi_{ij}\omega_{ij}^{2q-1},
\]
where $r \in [0,1]$ and $q \in [\frac{1}{2},1]$. If $r=0$ and $d_i=0$, we interpret $d_i^r$ as $0$. Similarly for $\omega_{ij}^{2q-1}$ and other such expressions below.

We define the graph gradient operator $(\nabla:\mathcal{V} \to \mathcal{E})$ by, for all $(i,j) \in E$, 
\[
(\nabla u)_{ij} := \omega_{ij}^{1-q}(u_j - u_i).
\] 
We define the graph divergence operator $(\text{div}: \mathcal{E} \to \mathcal{V})$ as the adjoint of the gradient, and a graph Laplacian operator $(\Delta_r: \mathcal{V} \to \mathcal{V})$ as the graph divergence of the graph gradient: for all $i\in V$,
\begin{equation}\label{eq:graphLaplacian}
{(\text{div}\varphi)}_i := \frac{1}{2} d_i^{-r} \sum_{j \in V} \omega_{ij}^q(\varphi_{ji} - \varphi_{ij}), \qquad
{(\Delta_r u)}_i: = (\text{div}(\nabla u))_i = d_i^{-r}\sum_{j \in V} \omega_{ij}(u_i - u_j).
\end{equation}
We note that the choices $r=0$ and $r=1$ lead to $\Delta_r$ being the unnormalised graph Laplacian and random walk graph Laplacian, respectively \cite{mohar1991,von2007}. Hence it is useful for us to explicitly incorporate $r$ in the notation $\Delta_r$ for the graph Laplacian.

In analogy with the graph gradient, divergence, and Laplacian, we now define their `signless' counterparts. We define the signless gradient operator $(\nabla^+:\mathcal{V} \to \mathcal{E})$ by, for all $(i,j) \in E$,
\[
(\nabla^+ u)_{ij} := \omega_{ij}^{1-q} (u_j + u_i).
\] 
Then we define the signless divergence operator $(\text{div}^+: \mathcal{E} \to \mathcal{V})$ to be the adjoint of the signless gradient, and the signless Laplacian operator $(\Delta_r^+: \mathcal{V} \to \mathcal{V})$ as the signless divergence of the signless gradient\footnote{In some papers the space $\mathcal{E}$ is defined as the space of all {\it skew-symmetric} edge functions. We do not require the skew-symmetry condition here, hence $\nabla^+u \in \mathcal{E}$, having $\text{div}^+$ act on $\nabla^+ u$ is consistent with our definitions, and $\text{div}^+ \varphi$ is not identically equal to $0$ for all $\varphi\in \mathcal{E}$.}: for all $i\in V$,
\[
{(\text{div}^+\varphi)}_i := \frac{1}{2} d_i^{-r} \sum_{j \in V} \omega_{ij}^q(\varphi_{ji} + \varphi_{ij}), \qquad
(\Delta_r^+ u)_i: = (\text{div}^+(\nabla^+ u))_i = d_i^{-r}\sum_{j \in V} \omega_{ij}(u_i + u_j).
\]

By definition we have
\[
\langle \nabla u,\phi \rangle_{\mathcal{E}} = \langle u,\textnormal{div} \: \phi \rangle_{\mathcal{V}}, \qquad \langle \nabla^+ u,\phi \rangle_{\mathcal{E}} = \langle u,\textnormal{div}^+\phi \rangle_{\mathcal{V}}.
\]

\begin{proposition}\label{prop:selfadjoint}
The operators $\Delta_r: \mathcal{V}\to \mathcal{V}$ and $\Delta_r^+: \mathcal{V}\to \mathcal{V}$ are self-adjoint and positive-semidefinite.
\end{proposition}
\begin{proof}
Let $u,v\in \mathcal{V}$. Since $\langle u,\Delta_r v \rangle_{\mathcal{V}} = \langle \nabla u,\nabla v \rangle_{\mathcal{E}} = \langle \Delta_r u,v \rangle_{\mathcal{V}}$ and $\langle u,\Delta_r^+ v \rangle_{\mathcal{V}} = \langle \nabla^+ u,\nabla^+ v \rangle_{\mathcal{E}} = \langle \Delta_r^+ u,v \rangle_{\mathcal{V}}$, the operators are self-adjoint. Positive-semidefiniteness follows from $\langle u,\Delta_r u \rangle_{\mathcal{V}} = \langle \nabla u, \nabla u \rangle_{\mathcal{E}} \geq 0$ and $\langle u,\Delta_r^+ u \rangle_{\mathcal{V}} = \langle \nabla^+ u, \nabla^+ u \rangle_{\mathcal{E}} \geq 0$.
\end{proof}

In the literature a third graph Laplacian is often used, besides the unnormalised and random walk graph Laplacians. This symmetrically normalised graph Laplacian \cite{chung1997} is defined by, for all $i\in V$,
\[
(\Delta_{s}u)_i := \frac{1}{\sqrt{d_i}}\sum_{j \in V}\omega_{ij}\left(\frac{u_i}{\sqrt{d_i}} - \frac{u_j}{\sqrt{d_j}}\right).
\]

This Laplacian cannot be obtained by choosing a suitable $r$ in the framework we introduced above, but will be useful to consider in practical applications. Analogously, we define the signless symmetrically normalised graph Laplacian by, for all $i\in V$,
\[
(\Delta_{s}^+u)_i := \frac{1}{\sqrt{d_i}}\sum_{j \in V}\omega_{ij}\left(\frac{u_i}{\sqrt{d_i}} + \frac{u_j}{\sqrt{d_j}}\right).
\]

There is a canonical way to represent a function $u\in \mathcal{V}$ by a vector in $\R^n$ with components $u_i$. The operators $\Delta_r$ and $\Delta_r^+$ can then be represented by the $n\times n$ matrices $L_r := D^{1-r} - D^{-r}A$ and $L^+_r := D^{1-r} + D^{-r}A$, respectively. Here $D$ is the degree matrix, i.e. the diagonal matrix with diagonal entries $D_{ii}:=d_i$, and $A$ is the weighted adjacency matrix with entries $A_{ij} := \omega_{ij}$.

 Similarly the operators $\Delta_s$ and $\Delta_s^+$ are then represented by $L_s := I - D^{-1/2}AD^{-1/2}$ and $L_s^+ := I + D^{-1/2}AD^{-1/2}$, respectively, where $I$ denotes the $n\times n$ identity matrix. Any eigenvalue-eigenvector pair $(\lambda,v)$ of $L_r$, $L_r^+$, $L_s$, $L_s^+$ corresponds via the canonical representation to an eigenvalue-eigenfunction pair $(\lambda, \phi)$ of $\Delta_r$, $\Delta_r^+$, $\Delta_s$, $\Delta_s^+$, respectively. We refer to the eigenvalue-eigenvector pair $(\lambda,v)$ as an eigenpair.

For a vertex set $S\subset V$, we define the indicator function (or characteristic function)
\[ \chi_S:=
\begin{cases}
1, & \text{if} \quad i \in S,\\
0, & \text{if} \quad i \notin S.\\
\end{cases}\]

We define the inner product norms $\|u\|_{\mathcal{V}} := \sqrt{{\langle u,u \rangle}_{\mathcal{V}}}, \: \|\phi\|_{\mathcal{E}} := \sqrt{{\langle \phi,\phi \rangle}_{\mathcal{E}}}$  which we use to define the Dirichlet energy and signless Dirichlet energy,
\[
\frac12 \|\nabla u\|_{\mathcal{E}}^2 = \frac{1}{4}\displaystyle\sum_{i,j \in V} \omega_{ij}(u_i - u_j)^2 \quad \text{and} \quad
\frac12 \|\nabla^+ u\|_{\mathcal{E}}^2 =  \frac{1}{4}\displaystyle\sum_{i,j \in V} \omega_{ij}(u_i + u_j)^2.
\]
In particular we recognise that the graph Ginzburg-Landau functional $f_\e:\mathcal{V}\to \R$ from \eqref{eq:graphGL} and the signless graph Ginzburg-Landau functional $f_\e^+:\mathcal{V}\to \R$ from \eqref{eq:signlessgraphGL} can be written as
\[
f_\e(u) = \|\nabla u\|_{\mathcal{E}}^2 + \frac1\e \sum_{i\in V} W(u_i) \quad \text{and} \quad 
f_\e^+(u) = \|\nabla^+ u\|_{\mathcal{E}}^2 + \frac1\e \sum_{i\in V} W(u_i).
\]

It is interesting to note here an important difference between the functionals $f_{\e}$ and $f_{\e}^+$. Most of the results that are derived in the literature for $f_\e$ (such as the $\Gamma$-convergence results in \cite{van2012}) do not crucially depend on the specific locations of the wells of $W$. For example, in $f_\e$ the wells are often chosen to be at $0$ and $1$, instead of at $-1$ and $1$. However, for $f_\e^+$ we have less freedom to choose the wells without drastically altering the properties of the functional. The wells have to be placed symmetrically with respect to $0$, because we want $(u_i+u_j)^2$ to be zero when $u_i$ and $u_j$ are located in different wells. In particular, we see that placing a well at $0$ would have the undesired consequence of introducing the trivial minimizer $u=0$. This points to a second, related, difference. Whereas minimization of $f_\e$ in the absence of any further constraints or additional terms in the functional leads to trivial minimizers of the form $u=c \chi_V$, where $c\in \R$ is one of the values of the wells of $W$ (so $c\in \{-1,1\}$ for our choice of $W$), minimizers of $f_\e^+$ are not constant, if the graph has more than one vertex. The following lemma gives the details.

\begin{lemma}
Let $G\in\mathcal{G}$ with $n\geq 2$, let $\e>0$, and let $u$ be a minimizer of $f_\e^+: \mathcal{V}\to\R$ as in \eqref{eq:signlessgraphGL}. Then $u$ is not a constant function.
\end{lemma}
\begin{proof}
Let $c\in\R$ and $i^*\in V$. Define the functions $u, \bar u\in \mathcal{V}$ by $u:=c\chi_V$ and
\[
\bar u_i :=\begin{cases} c, & \text{if } i\neq i^*,\\ -c, & \text{if } i=i^*.\end{cases}
\]
Since $W$ is an even function, we have $\sum_{i\in V} W(\bar u_i) = \sum_{i\in V} W(u_i)$. Moreover, since for all $j\in V$, $\omega_{i^*j}=0$ or $u_j=-u_{i^*}$, we have
\[
\|\nabla^+ \bar u\|_{\mathcal{E}}^2 = \frac12 \sum_{\substack{i\in V \\ i\neq i^*}} \sum_{j\in V} \omega_{ij} (2c)^2 < \frac12 \sum_{i,j\in V} \omega_{ij} (2c)^2 = \|\nabla^+ u\|_{\mathcal{E}}^2.
\]
The inequality is strict, because per assumption $G$ has no isolated nodes and thus there is a $j\in V$ such that $\omega_{i^*j}>0$. We conclude that $f_\e^+(\bar u) < f_\e^+(u)$, which proves that $u$ is not a minimizer of $f_\e^+$.
\end{proof}

We define the graph total variation $\textnormal{TV}: \mathcal{V}\to \R$ as
\begin{equation}\label{eq:graphTV}
\textnormal{TV}(u):= \max\{\langle u, \textnormal{div}\ \varphi\rangle_{\mathcal{V}}: \varphi\in \mathcal{E}, \forall i,j\in V\ |\varphi_{ij}|\leq 1\}
= \frac12 \sum_{i,j\in V} \omega_{ij}^q |u_i-u_j|.
\end{equation}
The second expression follows since the maximum in the definition is achieved by $\varphi = \textnormal{sgn}(\nabla u)$ \cite{van2014}. We can define an analogous (signless total variation) functional $\textnormal{TV}^+:\mathcal{V}\to \R$, using the signless divergence:
\[
\textnormal{TV}^+(u):= \max\{\langle u, \textnormal{div}^+\ \varphi\rangle_{\mathcal{V}}: \varphi\in \mathcal{E}, \forall i,j\in V\ |\varphi_{ij}|\leq 1\}.
\]
\begin{lemma}
Let $u\in \mathcal{V}$, then $\textnormal{TV}^+(u) = \frac12 \sum_{i,j\in V} \omega_{ij}^q |u_i+u_j|$.
\end{lemma}
\begin{proof}
Let $\varphi\in \mathcal{E}$ such that, for all $i,j\in V$, $|\varphi_{ij}|\leq 1$. We compute 
\begin{align*}
\langle u, \textnormal{div}^+\ \varphi\rangle_{\mathcal{V}} &= \frac12 \sum_{i,j\in V} \omega_{ij}^q u_i (\varphi_{ji}+\varphi_{ij}) = \frac12 \sum_{i,j\in V} \omega_{ij}^q \varphi_{ij} (u_i+u_j)\\
&\leq \frac12 \sum_{i,j\in V} \omega_{ij}^q |\varphi_{ij}| |u_i+u_j|
\leq \frac12 \sum_{i,j\in V} \omega_{ij}^q |u_i+u_j|.
\end{align*}
Moreover, since $\varphi = \textnormal{sgn}(\nabla^+u)$ is an admissable choice for $\varphi$ and
\[
\langle\textnormal{sgn}\left(\nabla^+u\right), \textnormal{div}^+\ \varphi\rangle_{\mathcal{V}}  = \frac12 \sum_{i,j\in V} \omega_{ij}^q |u_i+u_j|,
\]
the result follows.
\end{proof}
Note that the total variation functional that was mentioned in Section~\ref{sec:GL} corresponds to the choice $q=1$ in \eqref{eq:graphTV}. This is the relevant choice for this paper and hence from now on we will assume that $q=1$. Note that the choice of $q$ does not have any influence on the form of the graph (signless) Laplacians.

One consequence of the choice $q=1$ is that $\TV$ and $\TV^+$ are now closely connected to cut sizes: If $S\subset V$ and $C=S|S^c$ is the cut induced by $S$, then
\begin{equation}\label{eq:TVandcut}
\textnormal{TV}\left(\chi_S-\chi_{S^c}\right) = 2 \textnormal{TV}\left(\chi_S\right) = 2 s(C) \quad \text{and} \quad \textnormal{TV}^+\left(\chi_S-\chi_{S^c}\right) = \sum_{i,j\in V} \omega_{ij} - 2 s(C).
\end{equation}
We will give a precise definition of $s(C)$ in Definition~\ref{def:sizeofcut} below.

\begin{definition}

Let $G \in \mathcal{G}$. Then $G$ is bipartite if and only if there exist $A \subset V$, $B \subset V$, such that all the conditions below are satisfied:
\begin{itemize}
\item $A \cup B= V$, 
\item $A \cap B= \emptyset$, and
\item for all $(i,j)\in E$, $i\in A$ and $j\in B$, or $i\in B$ and $j\in A$.
\end{itemize}
In that case we say that $G$ has a bipartition $(A,B)$.
\end{definition} 

\begin{definition}
An Erd\"{o}s-R\'enyi graph $G(n,p)$ is a realization of a random graph generated by the  Erd\"{o}s-R\'enyi  model, i.e. it is an unweighted, undirected, simple graph with $n$ nodes, in which, for all unordered pairs $\{i,j\}$ of distinct $i,j \in V$, an edge $(i,j) \in E$ has been generated with probability $p\in [0,1]$.
\end{definition}


\subsection{Spectral properties of the (signless) graph Laplacians}

We consider the Rayleigh quotients for $\Delta_r$ and $\Delta_r^+$ defined, for $u\in V$, as
\begin{align*}
R(u) &:= \frac{{\langle u,\Delta_r u\rangle}_{\mathcal{V}}}{{\|u\|}_{\mathcal{V}}^2} = \frac{\|\nabla u\|_{\mathcal{E}}^2}{{\|u\|}_{\mathcal{V}}^2} = \frac{\frac{1}{2}\sum_{i,j \in V} \omega_{ij} (u_i - u_j)^2}{\sum_{i \in V} d_i^r u_i^2},\\
R^+(u) &:= \frac{{\langle u,\Delta_r^+ u\rangle}_{\mathcal{V}}}{{\|u\|}_{\mathcal{V}}^2} = \frac{\|\nabla^+ u\|_{\mathcal{E}}^2}{{\|u\|}_{\mathcal{V}}^2} = \frac{\frac{1}{2}\sum_{i,j \in V} \omega_{ij} (u_i + u_j)^2}{\sum_{i \in V} d_i^r u_i^2},
\end{align*}
respectively. By Proposition~\ref{prop:selfadjoint}, $\Delta_r$ and $\Delta_r^+$ are self-adjoint and positive-semidefinite operators on $\mathcal{V}$, so their eigenvalues will be real and non-negative. The eigenvalues of $\Delta_r$ and $\Delta_r^+$ are linked to the extremal values of their Rayleigh quotients by the min-max theorem \cite{courant1965,golub2012}. In particular, if we denote by $0\leq \lambda_1 \leq \ldots \leq \lambda_n$ the (possibly repeated) eigenvalues of $\Delta^+$, then $\lambda_1=\underset{u_1 \in \mathcal{V} \backslash \{0\}}{\mathrm{min}} \: R^+(u_1)$ and $\lambda_n=\underset{u_n \in \mathcal{V} \backslash \{0\}}{\mathrm{max}} \: R^+(u_n)$.
 
In the following proposition we extend a well-known result for the graph Laplacians \cite{von2007} to include signless graph Laplacians.
\begin{proposition}\label{prop:eigenpairs}
Let $r\in [0,1]$. The following statements are equivalent:
\begin{enumerate}
\item $\lambda$ is an eigenvalue of $L_1$ with corresponding eigenvector $v$;
\item $\lambda$ is an eigenvalue of $L_s$ with corresponding eigenvector $D^{1/2}v$;
\item $2-\lambda$ is an eigenvalue of $L_1^+$ with corresponding eigenvector $v$;
\item $2-\lambda$ is an eigenvalue of $L_s^+$ with corresponding eigenvector $D^{1/2} v$;
\item $\lambda$ and $v$ are solutions of the generalized eigenvalue problem $L_r v = \lambda D^{1-r} v$.
\end{enumerate}
\end{proposition}
\begin{proof}
For $r=1$ the matrix representations of the graph Laplacian and signless graph Laplacian satisfy $L^+_1 = I + D^{-1}A = 2I - (I - D^{-1}A) = 2I - L_1$. Hence $\lambda$ is an eigenvalue of $L_1$ with corresponding eigenvector $v$ if and only if $2 - \lambda$ is an eigenvalue of $L^+_1$ with the same eigenvector.

Because $L_s = D^{1/2} L_1 D^{-1/2}$, $\lambda$ is an eigenvalue of $L_1$ with eigenvector $v$ if and only if $\lambda$ is an eigenvalue of $L_s$ with eigenvector $D^{1/2}v$. Moreover, since $L_s^+ = 2I - L_s$, we have that $2-\lambda$ is an eigenvalue of $L_s^+$ with eigenvector $D^{1/2}v$ if and only if $\lambda$ is an eigenvalue of $L_s$ with eigenvalue $D^{1/2}v$.

Finally, for $r\in [0,1]$, we have $L_r = D^{1-r}L_1$, hence $\lambda$ is an eigenvalue of $L_1$ with corresponding eigenvector $v$ if and only if $L_r v = \lambda D^{1-r} v$.
\end{proof}

Inspired by Proposition~\ref{prop:eigenpairs}, we define, for a given graph $G\in \mathcal{G}$ and node subset $S\subset V$, the rescaled indicator function $\tilde \chi_S \in \mathcal{V}$, by, for all $j\in V$,
\begin{equation}\label{eq:rescaledindicator}
\left(\tilde \chi_S\right)_j:= d^{\frac12}_j \left(\chi_S\right)_j.
\end{equation}

\begin{proposition}\label{prop:connected}
The graph $G=(V,E,\omega)\in \mathcal{G}$ has $k$ connected components if and only if $\Delta\in \{\Delta_r, \Delta_s\}$ ($r\in [0,1]$) has eigenvalue $0$ with algebraic and geometric multiplicity equal to $k$. In that case, the eigenspace corresponding to the $0$ eigenvalue is spanned by
\begin{itemize}
\item the indicator functions $\chi_{S_i}$,  if $\Delta=\Delta_r$, or 
\item the rescaled indicator functions $\tilde \chi_{S_i}$ (as in \eqref{eq:rescaledindicator}), if $\Delta=\Delta_s$.
\end{itemize}
Here the node subsets $S_i \subset V$, $i\in \{1, \ldots, k\}$, are such that each connected component of $G$ is the subgraph induced by an $S_i$. 
\end{proposition}
\begin{proof}
We follow the proof in \cite{von2007}. First we consider the case where $\Delta=\Delta_r$, $r\in [0,1]$. We note that $\Delta_r$ is diagonizable in the $\mathcal{V}$ inner product and thus the algebraic multiplicity of any of its eigenvalues is equal to its geometric multiplicity. In this proof we will thus refer to both simply as `multiplicity'.

For any function $u \in \mathcal{V}$ we have that
$\displaystyle
\langle u, \Delta_r u \rangle_{\mathcal{V}} = \frac{1}{2}\sum_{i,j \in V}\omega_{ij}(u_i - u_j)^2.
$
We have that $0$ is an eigenvalue if and only if there exists a $u\in \mathcal{V}\setminus\{0\}$ such that
\begin{equation}\label{eq:uDeltauzero}
\langle u, \Delta_r u \rangle_{\mathcal{V}} = 0.
\end{equation}
This condition is satisfied if and only if, for all $i,j\in V$ for which $\omega_{ij}>0$, $u_i=u_j$. 

Now assume that $G$ is connected (hence $G$ has $k=1$ connected component), then \eqref{eq:uDeltauzero} is satisfied if and only if, for all $i,j\in V$, $u_i = u_j$. Therefore any eigenfunction corresponding to the eigenvalue $\lambda_1 = 0$ has to be constant, e.g. $u = \chi_V$. In particular, the multiplicity of $\lambda_1$ is 1.


Now assume that $G$ has $k \geq 2$ connected components, let $S_i$, $i\in \{1, \ldots, k\}$ be the node sets corresponding to the connected components of the graph. Via a suitable reordering of nodes $G$ will have a graph Laplacian matrix of the form
\[
L_r =  \begin{pmatrix}
  L_r^{(1)} & 0 & \cdots & 0 \\
  0 & L_r^{(2)} & \cdots & 0 \\
  \vdots  & \vdots  & \ddots & \vdots  \\
  0 & 0 & \cdots & L_r^{(k)} 
 \end{pmatrix},
\]
where each matrix $L_r^{(i)}$, $i \in \{1,\dots,k\}$ corresponds to $\Delta_r$ restricted to the connected component induced by $S_i$. This restriction is itself a graph Laplacian for that component. Because each $L_r^{(i)}$ has eigenvalue zero with multiplicity $1$, $L_r$ (and thus $\Delta_r$) has eigenvalue 0 with multiplicity $k$. We can choose the eigenvectors equal to $\chi_{S_i}$ for $i \in \{1,\dots,k\}$ by a similar argument as in the $k=1$ case.

Conversely, if $\Delta_r$ has eigenvalue $0$ with multiplicity $k$, then $G$ has $k$ connected components, because if $G$ has $l\neq k$ connected components, then by the proof above the eigenvalue $0$ has multiplicity $l\neq k$.

For $\Delta_s$ we use Proposition~\ref{prop:eigenpairs} to find that the eigenvalues are the same as those of $\Delta_r$, with the corresponding eigenfunctions rescaled as stated in the result.
\end{proof}


\begin{proposition}\label{prop:signlessspectrum}
Let $G=(V,E,\omega)\in \mathcal{G}$ have $k$ connected components and let the node subsets $S_i\subset V$, $i\in \{1, \ldots, k\}$ be such that each connected component is the subgraph induced by one of the $S_i$. We denote these subgraphs by $G_i$. Let $\Delta^+ \in \{\Delta_r^+, \Delta_s^+\}$  ($r\in [0,1]$) and let $0\leq k'\leq k$. Then $\Delta^+$ has an eigenvalue equal to 0 with algebraic and geometric multiplicity $k'$ if and only if $k'$ of the subgraphs $G_i$ are bipartite. In that case, assume the labelling is such that $G_i$, $i\in \{1, \ldots, k'\}$ are bipartite with bipartition $(T_i, S_i\setminus T_i)$, where $T_i\subset S_i$. Then the eigenspace corresponding to the 0 eigenvalue is spanned by
\begin{itemize}
\item the indicator functions  $\chi_{T_i} - \chi_{S_i\setminus T_i}$, if $\Delta^+ = \Delta_r^+$, or
\item the rescaled indicator functions $\tilde \chi_{T_i}-\tilde \chi_{S_i\setminus T_i}$ (as in \eqref{eq:rescaledindicator}), if $\Delta^+=\Delta_s^+$.
\end{itemize}
\end{proposition}
\begin{proof}
First we consider the case where $\Delta^+ = \Delta_r^+, r \in [0,1]$. For any vector $u \in \mathcal{V}$ we have that
\[
\langle u, \Delta_r^+ u \rangle_{\mathcal{V}} = \frac{1}{2}\sum_{i,j \in V}\omega_{ij}(u_i + u_j)^2.
\]
Let $k=1$, then $\lambda_1 = 0$ is an eigenvalue if and only if there exists $u\in \mathcal{V}\setminus\{0\}$ such that $\langle u, \Delta_r^+ u \rangle_{\mathcal{V}} = 0$. This condition is satisfied if and only if, for all $i,j\in V$ for which $\omega_{ij}>0$ we have
\begin{equation}\label{eq:ui=-uj}
u_i=-u_j.
\end{equation} We claim that this condition in turn is satisfied if and only if $G$ is bipartite. To prove the `if' part of that claim, assume $G$ is bipartite with bipartition $(A,A^c)$ for some $A\subset V$, and define $u\in \mathcal{V}$ such that $u|_A=-1$ and $u|_{A^c}=1$. To prove the `only if' statement, assume $G$ is not bipartite, then there exists an odd cycle in $G$ \cite[Theorem 1.4]{bollobas2013}. Let $i\in V$ be a vertex on this cycle, then by applying condition \eqref{eq:ui=-uj} to all the vertices of the cycle, we find $u_i=-u_i=0$. Since $G$ is connected, it now follows, by applying condition \eqref{eq:ui=-uj} to all vertices in $V$, that $u=0$, which is a contradiction.

The argument above also shows that, if $G$ is bipartite with bipartition $(A,A^c)$, then any eigenfunction corresponding to $\lambda_1=0$ is proportional to $u = \chi_A - \chi_{A^c}$. Therefore the eigenvalue 0 has geometric multiplicity 1. Since $\Delta_r^+$ is diagonizable in the $\mathcal{V}$ inner product the algebraic multiplicity of $\lambda_1$ is be equal to its geometric multiplicity.

Now let $k \ge 2$ and let $S_i$, $i\in \{1, \ldots, k\}$ be the node sets corresponding to the connected components of the graph. Via a suitable reordering of nodes the graph $G$ will have a signless Laplacian matrix of the form
\[
L_r^+ =  \begin{pmatrix}
  L^{(1)+} & 0 & \cdots & 0 \\
  0 & L^{(2)+} & \cdots & 0 \\
  \vdots  & \vdots  & \ddots & \vdots  \\
  0 & 0 & \cdots & L^{(k)+} 
 \end{pmatrix},
\]
where each matrix $L^{(i)+}$, $i \in \{1,\dots,k\}$ corresponds to $\Delta_r^+$ restricted to the connected component induced by $S_i$. This restriction is itself a signless Laplacian for that connected component. Hence, we can apply the case $k=1$ to each component separately to find that the (algebraic and geometric) multiplicity of the eigenvalue 0 of $\Delta_r^+$  is equal to the number of connected components which are also bipartite. 

If $G$ has $k'\leq k$ connected components which are also bipartite, then, without loss of generality, assume that these components correspond to $S_i$, $i\in \{1, \ldots, k'\}$. Then the corresponding eigenspace is spanned by functions $u^{(i)} = \chi_{T_i} - \chi_{S_i\setminus T_i} \in \mathcal{V}$, $i\in \{1, \ldots, k'\}$, where $T_i\subset S_i$ and $(T_i, S_i\setminus T_i)$ is the bipartition of the bipartite component induced by $S_i$. 

For $\Delta_s^+$ we use Proposition~\ref{prop:eigenpairs} to find the appropriately rescaled eigenfunctions as given in the result.
\end{proof}

\begin{corollary}
The eigenvalues of $\Delta_1$, $\Delta^+_1$ , $\Delta_s$, and $\Delta_s^+$ are in $[0,2]$.
\end{corollary}
\begin{proof}
For $\Delta_1$ the proof can be found in \cite[Lemma 2.5]{van2014}. For completeness we reproduce it here. By Proposition~\ref{prop:selfadjoint} we know that $\Delta_1$ has non-negative eigenvalues. The upper bound is obtained by maximizing the Rayleigh quotient $R(u)$ over all nonzero $u\in \mathcal{V}$. Since $(u_i - u_j)^2 \leq 2(u_i^2 + u_j^2)$ we have that
\begin{align*}
\underset{u \in \mathcal{V}\setminus\{0\}}{\mathrm{max}}\ R(u) &= \underset{u \in \mathcal{V}\setminus\{0\}}{\mathrm{max}} \frac{\frac{1}{2}\sum_{i,j \in V} \omega_{ij} (u_i - u_j)^2}{\sum_{i \in V} d_i u_i^2} \leq \underset{u \in \mathcal{V}\setminus\{0\}}{\mathrm{max}} \frac{2 \sum_{i \in V}d_i u_i^2}{\sum_{i \in V}d_i u_i^2} = 2.\\
\end{align*}

From Proposition~\ref{prop:eigenpairs} it then follows that the eigenvalues of the other operators are in $[0,2]$ as well.
\end{proof}


\section{The Max-Cut problem and Goemans-Williamson algorithm}\label{sec:MCGW}

\subsection{Maximum cuts}

In order to identify candidate solutions to the Max-Cut problem with node functions in $\mathcal{V}$, we define the subset of binary $\{-1,1\}$-valued node functions,
\[
\mathcal{V}^b := \{u \in \mathcal{V}: \forall \: i \in V, u_i \in \{-1,1\}\}.
\]
For a given function $u \in \mathcal{V}^b$ we define the sets $V_k := \{i \in V, u_i = k\}$ for $k \in \{-1,1\}$. We say that the partition $C = V_{-1}|V_1$ is the cut induced by $u$. We define the set of all possible cuts, $\mathcal{C} := \{C: \text{there exists a } u\in \mathcal{V}^b \text{ such that } u \text{ induces the cut } C\}$.

\begin{definition}\label{def:sizeofcut}
Let $G = (V,E,\omega)\in \mathcal{G}$ and let $V_1$ and $V_{-1}$ be two disjoint subsets of $V$. The size of the cut $C = V_{-1}|V_1$ is
\[
s(C) := \sum_{\substack{i \in V_{-1}\\ j \in V_1}}\omega_{ij}.
\]
A maximum cut of $G$ is a cut $C^*\in \mathcal{C}$ such that, for all cuts $C\in\mathcal{C}$, $s(C) \leq s(C^*)$. The size of the maximum cut is
\[
\mc{G} := \underset{C\in \mathcal{C}}\max\ s(C) .
\]
\end{definition}
Note that if the cut $C$ in Definition~\ref{def:sizeofcut} is induced by $u \in \mathcal{V}^b$, then
\begin{equation}\label{eq:cutLaplacian}
s(C) = \frac{1}{4}\langle u,\Delta_r u \rangle_{\mathcal{V}}.
\end{equation}
Moreover, if $C=\emptyset|V_1$ or $C=V_{-1}|\emptyset$, then $s(C)=0$. 

\begin{definition}[Max-Cut problem]
Given a simple, undirected graph $G=(V,E,\omega)\in \mathcal{G}$, find a maximum cut for $G$.
\end{definition}
For a given $G\in \mathcal{G}$ we have $|E|<\infty$, hence a maximum cut for $G$ exists, but note that this maximum cut need not be unique.

The cardinality of the set $\mathcal{V}^b$ is equal to the total number of ways a set of $n$ elements can be partitioned into two disjoint subsets, i.e. $|\mathcal{V}^b| = 2^n$. This highlights the difficulty of finding $\mc{G}$ as $n$ increases. It has been proven that the Max-Cut problem is NP-hard \cite{garey1979}. Obtaining a performance guarantee of $\frac{16}{17}$ or better is also NP-hard \cite{trevisan2000}. The problem of determining if a cut of a given size exists on a graph is NP-complete \cite{karp}.


\subsection{The Goemans-Williamson algorithm}

The leading algorithm for polynomial time Max-Cut approximation is the Goemans-Williamson (GW) algorithm \cite{goemans1995}, which we present in detail below in Algorithm~\ref{alg:GW}. A problem equivalent to the Max-Cut problem is to find a maximizer which achieves
\[
\underset{u}{\mathrm{max}} \: \frac{1}{2} \sum_{i,j}\omega_{ij}(1 - u_iu_j) \quad \mathrm{subject \: to } \: \forall i \in V, u_i \in \{-1,1\}.
\]
The GW algorithm solves a relaxed version of this integer quadratic program, by allowing $u$ to be an $n$-dimensional vector with unit Euclidean norm. In \cite{goemans1995} it is proved that the $n$-dimensional vector relaxation is an upper bound on the original integer quadratic program. This relaxed problem is equivalent to finding a maximizer which achieves
\begin{equation}\label{eq:ZP}
Z^*_P := \underset{Y}{\mathrm{max}} \: \frac{1}{2}\sum_{i,j\in V, i < j}\omega_{ij}(1 - y_{ij}),
\end{equation}
where the maximization is over all $n$ by $n$ real positive-semidefinite matrices $Y=(y_{ij})$ with ones on the diagonal. This semidefinite program has an associated dual problem of finding a minimizer which achieves
\begin{equation}\label{eq:ZD}
Z^*_D := \frac{1}{2}\sum_{i,j \in V}\omega_{ij} + \frac{1}{4}\underset{\gamma\in \R^n}{\mathrm{min}}\sum_{i \in V} \gamma_i,
\end{equation}
subject to $A + \mathrm{diag}(\gamma)$ being positive-semidefinite, where $A$ is the adjacency matrix of $G$ and $\mathrm{diag}(\gamma)$ is the diagonal matrix with diagonal entries $\gamma_i$.

As mentioned in Section~\ref{sec:maximumcut}, Algorithm~\ref{alg:GW} is proven to have a performance guarantee of $0.878$. In Algorithm~\ref{alg:GW} below, we use the unit sphere $S_n := \{x\in \R^n: \|x\| = 1\}$, where $\|\cdot\|$ denotes the Euclidean norm on $\R^n$. For vectors $w, \tilde w \in \R^n$, $w\cdot \tilde w$ denotes the Euclidean inner product.

\begin{asm}{(GW)}
\KwData{The weighted adjacency matrix $A$ of a graph $G\in \mathcal{G}$, and a tolerance $\nu$.}

\textbf{ Relaxation step:} Use semidefinite programming to find approximate solutions $\tilde Z^*_P$ and $\tilde Z^*_D$ to \eqref{eq:ZP} and \eqref{eq:ZD}, respectively, which satisfy $|\tilde Z^*_P - \tilde Z^*_D|< \nu$. Use an incomplete Cholesky decomposition on the matrix $Y$ that achieves $\tilde Z^*_P$ in \eqref{eq:ZP} to find an approximate solution to 
\[
w^* \in \underset{w \in (S_n)^n}\argmax\ \frac{1}{2} \sum_{\substack{1\leq i, j, \leq n\\ i<j}} \omega_{ij}(1 - w_i\cdot w_j).
\]

\medskip
\textbf{ Hyperplane step:} Let $r\in S_n$ be a random vector drawn from the uniform distribution on $S_n$. Define the cut $C:=V_{-1}|V_1$, where
\[
V_1 := \{i \in V | w_i \cdot r \geq 0\} \quad \text{ and } \quad  V_{-1}: = V \setminus V_1.
\]

\caption{\label{alg:GW} The Goemans-Williamson algorithm}
\end{asm}

Other polynomial time Max-Cut approximation algorithms can be found in \cite{bylka1999,trevisan2012}. Because of the high proven performance guarantee of \ref{alg:GW}, we focus on comparing our algorithm against it. In \cite{trevisan2012} the authors use the eigenvector corresponding to the smallest eigenvalue of $\Delta_0^+$, showing that thresholding this eigenvector in a particular way achieves a Max-Cut performance guarantee of $\beta = 0.531$, which with further analysis was improved to $\beta = 0.614$ \cite{soto2015}. Algorithms which provide a solution in polynomial time exist if the graph is planar \cite{hadlock1975}, if the graph is a line graph \cite{guruswami1999}, or if the graph is weakly bipartite \cite{grotschel1981}. Comparing against \cite{trevisan2012,hadlock1975,guruswami1999,grotschel1981} is a topic of future research.


\section{$\Gamma$-convergence of $f_{\varepsilon}^+$}\label{sec:fe+}

In \eqref{eq:signlessgraphGL} we introduced the signless Ginzburg-Landau functional $f_\e^+: \mathcal{V}\to \R$. In this section we prove minimizers of $f_\e^+$ converge to solutions of the Max-Cut problem, using the tools of $\Gamma$-convergence \cite{braides2002}.

We need a concept of convergence in $\mathcal{V}$. Since we can identify $\mathcal{V}$ with $\R^n$ and all norms on $\mathbb{R}^n$ are topologicallly equivalent, the choice of a particular norm is not of great importance. For definiteness, however, we say that sequence $\{u_k\}_{k\in \N} \subset \mathcal{V}$ converges to a $u_{\infty} \in \mathcal{V}$  in $\mathcal{V}$ if and only if
$\|\hat{u}_k - \hat{u}_{\infty}\|_{\mathcal{V}} \to 0$ as $ k \to \infty$, where $\hat{u}_k,\hat{u}_{\infty} \in \mathbb{R}^n$ are the canonical vector representations of $u_k$, $u_\infty$, respectively.

We will prove that $f_\e^+$ $\Gamma$-converges to the functional $f_0^+: \mathcal{V} \to \R\cup\{+\infty\}$, which is defined as
\begin{equation}\label{eq:f0+}
f_0^+ (u) := 
 \begin{cases}
   \sum_{i,j \in V} \omega_{ij} |u_i + u_j|, &  \text{if } u \in \mathcal{V}^b, \\
   +\infty, &  \text{if } u \in \mathcal{V} \setminus \mathcal{V}^b.
  \end{cases}
\end{equation}

\begin{lemma}\label{lem:MaxCut}
Let $G\in \mathcal{G}$. For every $u\in \mathcal{V}^b$, let $C_u\in \mathcal{C}$ be the cut induced by $u$, then for all $u\in \mathcal{V}$,
\[
f_0^+(u) = \begin{cases} 2\sum_{i,j \in V} \omega_{ij} - 4s(C_u), &\text{if } u \in \mathcal{V}^b,\\
+\infty, &\text{if } u \in \mathcal{V} \setminus \mathcal{V}^b.
\end{cases}
\]
In particular, if $u^*\in \underset{u \in \mathcal{V}}\argmin\, f_\e^+(u)$, then $u^*\in \mathcal{V}^b$ and $C_{u^*}$ is a maximum cut of $G$.
\end{lemma}
\begin{proof}
Because, for $u\in \mathcal{V}^b$, $f_0^+(u)=2\textnormal{TV}^+(u)$ (with $q=1$), the result follows by \eqref{eq:TVandcut}.
\end{proof}

\begin{lemma}\label{lem:minfunc}
Let $G\in\mathcal{G}$ and $\e>0$. There exist minimizers for the functionals $f_\e^+: \mathcal{V}\to\R$ and $f_0^+: \mathcal{V}\to \R\cup\{+\infty\}$ from \eqref{eq:signlessgraphGL} and \eqref{eq:f0+}, respectively. Moreover, if $u\in \mathcal{V}$ is a minimizer of $f_0^+$, then $u\in \mathcal{V}^b$.
\end{lemma}
\begin{proof}
The potential $W$ satisfies a coercivity condition in the following sense. There exist a $C_1>0$ and a $C_2$ such that, for all $x\in\R$,
\begin{equation}\label{eq:Wcoerc}
|x|\geq C_1 \Rightarrow C_2(x^2-1)\leq W(x).
\end{equation}
Combined with the fact that $\|\nabla^+u\|_{\mathcal{E}} \geq 0$, this shows that $f_\e^+$ is coercive. Since $f_\e^+$ is a (multivariate) polynomial, it is continuous. Thus, by the direct method in the calculus of variations \cite[Theorem 1.15]{dalmaso1993} $f_\e^+$ has a minimizer in $\mathcal{V}$.

Since $n\geq 1$, $\mathcal{V}^b\neq \emptyset$ and thus $\underset{u\in\mathcal{V}}\inf\, f_0^+(u)<+\infty$. In particular, any minimizer of $f_0^+$ has to be in $\mathcal{V}^b$. Since $|\mathcal{V}^b|<\infty$ the minimum is achieved.
\end{proof}

\begin{lemma}\label{lem:Gammaconvergence}
Let $G\in \mathcal{G}$ and let $f_\e^+$ and $f_0^+$ be as in \eqref{eq:signlessgraphGL} and \eqref{eq:f0+}, respectively. Then $f_\e^+$ $\Gamma$-converges to $f_0^+$ as $\e \downarrow 0$ in the following sense: If $\{\e_k\}_{k\in\N}$ is a sequence of positive real numbers such that $\e_k\downarrow 0$ as $k\to \infty$ and $u_0\in \mathcal{V}$, then the following lower bound and upper bound conditions are satisfied:
\begin{itemize}
\item[(LB)] for every sequence $\{u_k\}_{k=1}^{\infty}\subset \mathcal{V}$ such that $u_k \rightarrow u_0$ as $k\to \infty$, it holds that $f_0^+(u_0) \leq \underset{k\to\infty}\liminf\,  f_{\e_k}^+(u_k)$;

\item[(UB)] there exists a sequence $\{u_k\}_{k=1}^{\infty}\subset \mathcal{V}$ such that $u_k\to u_0$ as $k\to\infty$ and $f_0^+(u_0)\geq \underset{k\to\infty}\limsup\, f_{\e_k}^+(u_k)$.
\end{itemize}
\end{lemma}
\begin{proof}
This proof is an adaptation of the proofs in \cite[Section 3.1]{van2012}.

Note that
\[
f_\e^+(u) = \frac12 \sum_{i,j\in V} \omega_{ij} (u_i+u_j)^2 + w_\e(u),
\]
where we define $w_\e: \mathcal{V} \to \R$ by
\[
w_{\varepsilon} (u) := \frac{1}{\varepsilon} \sum_{i \in V} W(u_i).
\]
First we prove that $w_\e$ $\Gamma$-converges to $w_0$ as $\e\downarrow 0$, where
\[
w_0(u) :=
 \begin{cases}
  0, & \text{if } u \in \mathcal{V}^b,\\
  +\infty, &\text{if } u\in \mathcal{V}\setminus\mathcal{V}^b.
  \end{cases}
\]
Let $\{\e_k\}_{k\in\N}$ is a sequence of positive real numbers such that $\e_k\downarrow 0$ as $k\to \infty$ and $u_0\in \mathcal{V}$.

(LB) Note that, for all $u \in \mathcal{V}$ we have $w_{\e}(u) \geq 0$. Let $\{u_k\}_{k=1}^{\infty}$ be a sequence such that $u_k \to u_0$ as $k \to \infty$. First we assume that $u_0 \in \mathcal{V}^b$, then
\[
w_0(u_0) = 0 \leq \underset{k\to\infty}\liminf\, w_{\e_k}(u_k).
\]
Next suppose that $u_0 \in \mathcal{V}\setminus \mathcal{V}^b$, then there is an $i\in V$ such that $(u_0)_i\not\in\{-1,1\}$. Since $u_k\to u_0$ as $k\to\infty$, for every $\eta>0$ there is an $N(\eta)\in \N$ such that for all $k\geq N(\eta)$ we have that $d_i^r |(u_0)_i - (u_k)_i| < \eta$. Define
\[
\bar\eta := \frac12 d_i^r \min\left\{|1-(u_0)_i|, |-1-(u_0)_i|\right\} > 0,
\]
then, for all $k\geq N(\bar\eta)$,
\[
|1-(u_k)_i| \geq \big| |1-(u_0)_i| - |(u_0)_i-(u_k)_i| \big| \geq \frac12 |1-(u_0)_i| >0.
\]
Similarly, for all $n\geq N(\bar\eta)$, $|-1-(u_k)_i| \geq \frac12 |-1-(u_0)_i| >0$. Hence, there is a $c>0$ such that, for all $k\geq N(\bar\eta)$, $|(u_k)_i| \leq 1-c$. Thus there is a $C>0$ such that, for all $k\geq N(\bar\eta)$, $W((u_k)_i) \geq C$. It follows that
\[
\underset{k\to\infty}\liminf\, w_{\e_k}(u_k) \geq \underset{k\to\infty}\liminf\, \frac{1}{\e_k} W((u_k)_i) = \infty = w_0(u_0).
\]
(UB) If $u_0\in \mathcal{V}\setminus\mathcal{V}^b$, then $w_0(u_0)$ and the upper bound condition is trivially satisfied. Now assume $u_0\in \mathcal{V}^b$. Define the sequence $\{u_k\}_{k=1}^{\infty}$ by, that for all $k \in \mathbb{N}$, $u_k = u_0$. Then, for all $k\in \N$, $w_{\e_k}(u_k) = 0$ and thus
$\displaystyle
\underset{k\to\infty}\limsup\, w_{\e_k}(u_k) =  0 =  w_0(u_0).
$ This concludes the proof that $w_\e$ $\Gamma$-converges to $w_0$ as $\e\downarrow 0$.

It is known that $\Gamma$-convergence is stable under continuous perturbations \cite[Proposition 6.21]{dalmaso1993}, \cite[Remark 1.7]{braides2002}; thus $w_\e + p$ $\Gamma$-converges to $w_0+p$ for any continuous $p:\mathcal{V} \to \R$. Since $u\mapsto \frac12 \sum_{i,j\in V} \omega_{ij} (u_i+u_j)^2$ is a polynomial and hence a continuous function on $\mathcal{V}$, we find that, as $\e\downarrow 0$, $f_\e^+$ $\Gamma$-converges to $g: \mathcal{V}\to \R\cup\{+\infty\}$, where
\[
g(u) := \frac12 \sum_{i,j\in V} \omega_{ij} (u_i+u_j)^2 + w_0(u).
\]
If $u\in \mathcal{V}\setminus\mathcal{V}^b$, then $g(u)=+\infty$. If $u\in \mathcal{V}^b$, then, for all $i,j\in V$, $(u_i+u_j)^2 = 2|u_i+u_j|$, hence
\[
\frac12 \sum_{i,j\in V} \omega_{ij} (u_i+u_j)^2 = \sum_{i,j \in V} \omega_{ij} |u_i+u_j|.
\] 
Thus $g=f_0^+$ and the theorem is proven.
\end{proof}


\begin{lemma}\label{lem:equicoercivity}
Let $G\in \mathcal{G}$ and let $f_\e^+$ be as in \eqref{eq:signlessgraphGL}. Let $\{\varepsilon_k\}_{k\in\N} \subset (0,\infty)$ be a sequence such that $\varepsilon_k \downarrow 0$ as $k \to \infty$, then the sequence $\{f_{\e_k}^+\}_{k\in\N}$ satisfies the following equi-coerciveness property: If $\{u_k\}_{k\in\N} \subset \mathcal{V}$ is a sequence such that there exists $C > 0$ such that, for all $k \in \mathbb{N}$, $f_{\varepsilon_k}^+ (u_k) < C.$, then there exists a subsequence $\{u_{k'}\}_{k'\in\N} \subset \{u_k\}_{k\in\N}$ and a $u_0 \in \mathcal{V}^b$ such that $u_{k'} \to u_0$ as $k \to \infty$.
\end{lemma}
\begin{proof}
This proof closely follows \cite[Section 3.1]{van2012}.

From the uniform bound $f_{\e_k}^+(u_k) < C$, we have that, for all $k\in\N$ and all $i\in V$, $0\leq W((u_k)_i) \leq C$. Because of the coercivity property \eqref{eq:Wcoerc} of $W$, the uniform bound on $W((u_k)_i)$ gives, for all $i\in V$, boundedness of $\{d_i^r(u_k)_i^2\}_{k\in\N}$ and thus $\{u_k\}_{k\in\N}$ is bounded in the $\mathcal{V}$-norm. The result now follows by the Bolzano-Weierstrass theorem.
\end{proof}

%
%
%



With the $\Gamma$-convergence and equi-coercivity results from Lemmas~\ref{lem:Gammaconvergence} and~\ref{lem:equicoercivity}, respectively, in place, we now prove that minimizers of $f_\e^+$ converge to solutions of the Max-Cut problem.

\begin{theorem}
Let $G\in \mathcal{G}$. Let $\{\e_k\}_{k\in\N} \subset (0,\infty)$ be a sequence such that $\e_k\downarrow 0$ as $k\to\infty$ and, for each $k\in \N$,  let $f_{\e_k}^+$ be as in \eqref{eq:signlessgraphGL} and let $u_{\e_k}$ be a minimizer of $f_{\e_k}^+$. Then there exists $u_0\in \mathcal{V}^b$ and a subsequence $\{u_{\e_{k'}}\}_{k'\in\N} \subset \{u_{\e_k}\}_{k\in\N}$,  such that $\|u_{\e_{k'}}-u_0\|_{\mathcal{V}} \to 0$ as $k'\to \infty$. Furthermore, $u_0 \in \underset{u\in \mathcal{V}}\argmin\ f_0^+(u)$, where $f_0^+$ is as in \eqref{eq:f0+}. In particular, if $C_{u_0} \in \mathcal{C}$ is the cut induced by $u_0$, then $C_{u_0}$ is a maximum cut of $G$.
\end{theorem}
\begin{proof}
It is a well-known result from $\Gamma$-convergence theory \cite[Corollary 7.20]{dalmaso1993}, \cite[Theorem 1.21]{braides2002} that the equi-coercivity property of $\{f_{\e_k}\}_{k\in\N}$ from Lemma~\ref{lem:equicoercivity} combined with the $\Gamma$-convergence property of Lemma~\ref{lem:Gammaconvergence} implies that $\underset{u\in\mathcal{V}}\min\, f_\e^+(u)$ converge to $\underset{u\in\mathcal{V}}\min\, f_0^+(u)$ and, up to taking a subsequence, minimizers of $f_\e^+$ converge to a minimizer of $f_0^+$.

By Lemma~\ref{lem:minfunc}, if $u_0 \in \underset{u \in \mathcal{V}}\argmin\, f_\e^+(u)$, then $u_0 \in \mathcal{V}^b$. By Lemma~\ref{lem:MaxCut}, the cut $C_{u_0}$ induced by $u_0$ is a maximum cut of $G$.
\end{proof}


\section{The signless MBO algorithm}\label{sec:signlessMBO}

\subsection{Algorithm}\label{sec:algorithm}

One way of attempting to find minimizers of $f_\e^+$ is via its gradient flow \cite{ambrosio2008}. This is, for example, the method employed in \cite{bertozzi2012} to find approximate minimizers of $f_\e$. In that case the gradient flow is given by a graph-based analogue of the Allen-Cahn equation \cite{allen1979}. To find the $\mathcal{V}$-gradient flow of $f_\e^+$ we compute the first variation of the functional $f_{\e}^+$: for $t\in \R$, $u,v\in \mathcal{V}$, we have
\[
\frac{d}{dt}f_{\e}^+(u+tv) |_{t=0} = \langle \Delta_r^+u,v \rangle_{\mathcal{V}} + \frac{1}{\e}\langle D^{-r}W' \circ u, v \rangle_{\mathcal{V}},
\]
where we used the notation $(D^{-r}W' \circ u)_i = d_i^{-r}W'(u_i)$. This leads to the following $\mathcal{V}$-gradient flow: for all $i\in V$,
\begin{equation}\label{eq:gradflow}
\begin{cases}
\frac{du_i}{dt} = -(\Delta_r^+ u)_i - \frac{1}{\e}d_i^{-r}W'(u_i),& \text{for } t>0,\\
u_i = (u_0)_i,& \text{for } t=0.
\end{cases}
\end{equation}
Since $f_\e^+$ is not convex, as $t \to \infty$ the solution of the $\mathcal{V}$-gradient flow is not guaranteed to converge to a global minimum, and can get stuck in local minimizers. 

In this paper we will not attempt to directly solve the gradient flow equation. That could be the topic of future research. Instead we will use a graph MBO type scheme, which we call the signless MBO algorithm. It is given in \ref{alg:signlessMBO}. Despite there currently not being any rigorous results on the matter, the outcome of this scheme is believed to approximate minimizers of $f_\e^+$. The original MBO scheme (or threshold dynamics scheme) in the continuum was introduced to approximate motion by mean curvature flow \cite{MBO1992,MBO1993}. It consists of iteratively applying ($N$ times) two steps: diffusing a binary initial condition for a time $\tau$ and then thresholding the result back to a binary function. In the (suitably scaled) limit $\tau\downarrow 0$, $N\to\infty$, solutions of this process converge to solutions of motion by mean curvature \cite{barles1995}. It is also known that solutions the continuum Allen-Cahn equation (in the limit $\e\downarrow 0$) converge to solutions of motion by mean curvature \cite{bronsard1991}. Whether something similar is true for the graph MBO scheme or graph Allen-Cahn equation \cite{van2014} or something analogous is true for the signless graph MBO scheme are as of yet open questions, but it does suggest that solutions of the MBO scheme (signless MBO scheme) could be closely connected to minimizers of $f_\e$ ($f_\e^+$). In practice, the graph MBO scheme has proven to be a fast and accurate method for tackling approximate minimization problems of this kind \cite{merkurjev2013, bertozzi2012}.

We see in \ref{alg:signlessMBO} that in the signless diffusion step the equation that is solved is the gradient flow equation from \eqref{eq:gradflow} without the double well potential term. Since we expect the double well potential term in \eqref{eq:gradflow} to force the solution to take values close to $\pm 1$, the signless diffusion step in \ref{alg:signlessMBO} is followed by a thresholding step. Note that, despite our choice of nomenclature, the signless graph `diffusion' dynamics is expected to be significantly different from standard graph diffusion.

\begin{asm}{(MBO+)}
\KwData{A signless graph Laplacian $\Delta^+ \in \{\Delta_0^+, \Delta_1^+, \Delta_s^+\}$ corresponding to a graph $G\in \mathcal{G}$, a signless diffusion time $\tau>0$, an initial condition $\mu^0:=\chi_{S_0} - \chi_{S_0^c}$ corresponding to a node subset $S_0 \subset V$, a time step $dt$, and a stopping criterion tolerance $\eta$.}
\KwOut{A sequence of functions $\{\mu^j\}_{j=0}^{N} \subset \mathcal{V}^b$ giving the signless MBO evolution of $\mu^0$, a sequence of corresponding cuts $\{C^j\}_{j=0}^N \subset \mathcal{C}$ and their sizes $\{s(C^j)\}_{j=0}^N \subset [0,\infty)$, with largest value $s^*$.}
\For{$ j = 1 \ \KwTo \ $ \textnormal{stopping criterion is satisfied,}}{
\textbf{ Signless diffusion step:} Compute $u^*(\tau)$, where $u^*\in \mathcal{V}$ is the solution of the initial value problem
\begin{equation}\label{eq:signlessdiffusion}
\begin{cases}
\frac{du(t)}{dt} = -\Delta^+ u(t),& \text{for } t>0,\\
u(0) = \mu^j.&
\end{cases}
\end{equation}

\medskip

\textbf{ Threshold step:} Define $\mu^j \in \mathcal{V}^b$ by, for $i\in V$,
\begin{equation}\label{eq:threshold}
\mu^j_i := T(u^*_i(\tau)) := \begin{cases}
1, &\text{if } u^*_i(\tau) > 0,\\
-1, &\text{if } u^*_i(\tau) \leq 0.
\end{cases}
\end{equation}

Define the cut $C^j:=V_{-1}^j|V_1^j$, where $V_{\pm 1}^j := \{i\in V: \mu^{j}_i = \pm 1\}$ and compute $s(C^j)$.

Set $N=j$.

\If {
$\frac{\|\mu^j - \mu^{j-1}\|_2^2}{\|\mu^j\|_2^2} < \eta$}{ Stop}
}

\textbf{ Find the largest cut size: } Set $s^* := \max_{1\leq j \leq N} s(C^j)$.

%
%


\caption{\label{alg:signlessMBO} The signless graph MBO algorithm}
\end{asm}

In Figures~\ref{fig:EnergyAS8} and \ref{fig:EnergyGNutella} we show the minimization of $f_{\e}^+$ using \ref{alg:signlessMBO} with the spectral method (which is explained in Section~\ref{sec:spectral}) on the AS8 graph and the GNutella09 graph (see Section~\ref{sec:scale}). The \ref{alg:signlessMBO} iteration numbers $j$ are indicated along the $x$-axis. The $y$-axis shows the value of $f_{\e}^+(\mu^j)$. What we see in both figures is that the overall tendency is for the \ref{alg:signlessMBO} algorithm to decrease the value of $f_{\e}^+(\mu^j)$, however, in some iterations the value increases. This is why in \ref{alg:signlessMBO} we output the cut size which is largest among all iterations computed and use that as the final output, if it outperforms the cut $C$ which \ref{alg:signlessMBO} returns. Alternatively, in order to save on computing memory, one could also keep track of the largest cut size found so far in each iteration and discard the other cut sizes, or accept the final cut size $s(C^N)$ as approximation to $s^*$ . The result we report in this paper are all based on the output $s^*$.

In our experiments we choose the stopping criterion tolerance $\eta = 10^{-8}$.

\begin{figure}
\centering
\begin{subfigure}{.5\textwidth}
  \centering
  \includegraphics[width=7.5cm]{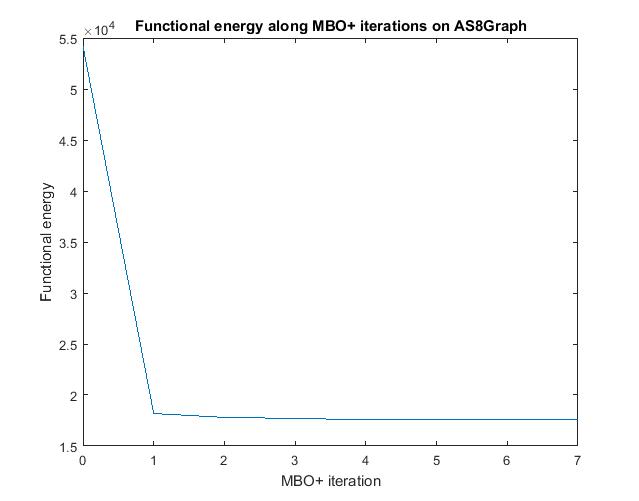}
  \label{fig:Energy1}
\end{subfigure}%
\begin{subfigure}{.5\textwidth}
  \centering
  \includegraphics[width=7.5cm]{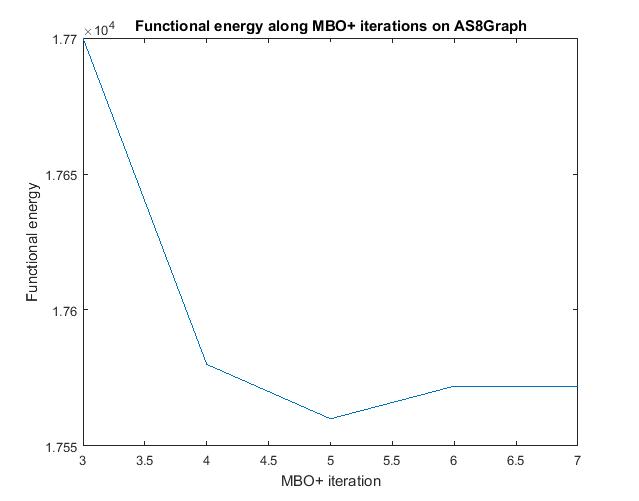}
  \label{fig:Energy3}
\end{subfigure}
\caption{The value $f_{\e}^+(\mu^j)$ as a function of the iteration number $j$ in the \ref{alg:signlessMBO} scheme on AS8Graph, using the spectral method and $\Delta_1^+$, with $K=100$, and $\tau = 20$. The left hand plot shows the initial condition and all iterations of the \ref{alg:signlessMBO} scheme on AS8Graph, where as the right hand plot displays the 3rd to the final iterations of the \ref{alg:signlessMBO} scheme on AS8Graph.} \label{fig:EnergyAS8}
\end{figure}

\begin{figure}
\centering
\begin{subfigure}{.5\textwidth}
\centering
  \includegraphics[width=7.5cm]{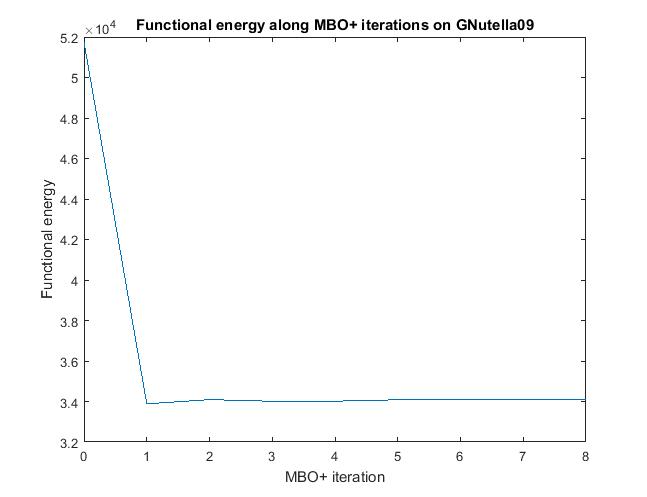}
  \label{fig:Energy2}
\end{subfigure}%
\begin{subfigure}{.5\textwidth}
\centering
  \includegraphics[width=7.5cm]{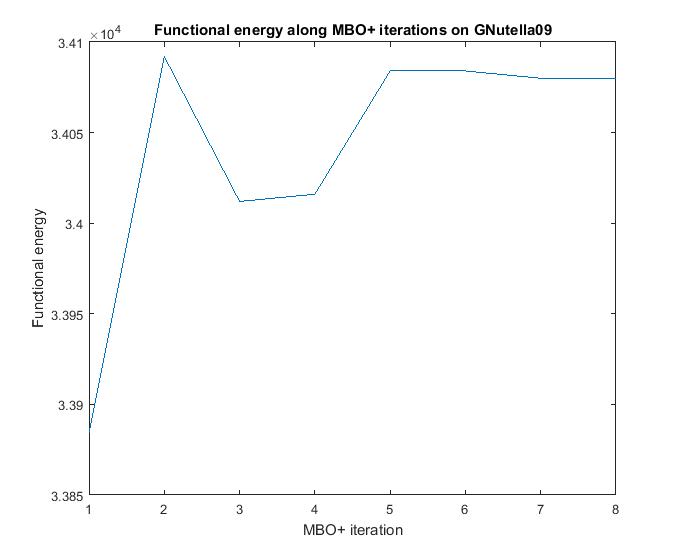}
  \label{fig:Energy4}
\end{subfigure}
\caption{The value $f_{\e}^+(\mu^j)$ as a function of the iteration number $j$ in the \ref{alg:signlessMBO} scheme on the GNutella09 graph, using the spectral method and $\Delta_1^+$, with $K=100$, and $\tau = 20$. The left hand plot shows the initial condition and all iterations of the \ref{alg:signlessMBO} scheme on GNutella09, where as the right hand plot displays all iterations of the \ref{alg:signlessMBO} scheme on GNutella09, without the initial condition.}\label{fig:EnergyGNutella}
\end{figure}


\subsection{Spectral decomposition method}\label{sec:spectral}

In this paper we will compare two implementations of the \ref{alg:signlessMBO} algorithm, which differ in the way they solve \eqref{eq:signlessdiffusion} for $t\in [0,\tau]$. In the next section we consider an explicit Euler method, but first we discuss a spectral decomposition method.
 In order to solve \eqref{eq:signlessdiffusion} we use spectral decomposition of the signless graph Laplacian $\Delta^+ \in \{\Delta^+_0, \Delta^+_1, \Delta^+_s\}$. Let $\lambda_k\geq 0$, $k\in \{1, \ldots, n\}$ be the eigenvalues of $\Delta^+$. We assume $\lambda_1 \leq \lambda_2 \leq \ldots \lambda_n$ and list eigenvalues multiple times according to their multiplicity. Let $\phi^k\in \mathcal{V}$ be an eigenfunction corresponding to $\lambda_k$, chosen such that $\{\phi_k\}_{k=1}^{n}$ is a set of orthonormal functions in $\mathcal{V}$. We then use the decomposition
\begin{equation}\label{eq:decomposition}
u^*(\tau) = \sum_{k=1}^n e^{-\lambda_k \tau}\langle \phi^k, u(0) \rangle_{\mathcal{V}} \: \phi^k
\end{equation}
to solve \eqref{eq:signlessdiffusion}.

For $\Delta_s^+$ we use the Euclidean inner product instead of the $\mathcal{V}$ inner product in \eqref{eq:decomposition}, because the Laplacian $\Delta_s^+$ is not of the form as given in \eqref{eq:graphLaplacian}. The optimal choice for $\tau$ with respect to the cut size obtained by \ref{alg:signlessMBO} is a topic for future research. Based on trial and error, we decided to use $\tau=20$ in the results we present in Section~\ref{sec:results}, when using $\Delta_1^+$ or $\Delta_s^+$ as our operator. We use $\tau = \frac{40}{\lambda_n}$ when using $\Delta_0^+$ as our operator, where $\lambda_n$ is the largest eigenvalue of $\Delta_0^+$. The division of $\tau$ by half of the largest eigenvalue of $\Delta_0^+$ is justified in Section~\ref{sec:pinning condition}. In Section~\ref{sec:parameters} we investigate how cut sizes change with varying $\tau$.

A computational advantage of the spectral decomposition method is that we do not necessarily need to use all of the eigenvalues and eigenfunctions of the signless Laplacian. We can use only the $K$ eigenfunctions corresponding to the smallest eigenvalues in our decomposition \eqref{eq:decomposition}. To be explicit, doing this replaces $n$ in \eqref{eq:decomposition} by $K$. In Section~\ref{sec:parameters} we show how increasing $K$ beyond a certain point has little effect on the size of the cut obtained by \ref{alg:signlessMBO} for three examples. We refer to using the $K$ eigenfunctions corresponding to the smallest eigenvalues in the decomposition as spectral truncation.

By Proposition~\ref{prop:eigenpairs}, we can compute the $K$ smallest eigenvalues $\lambda_k$ ($k\in \{1, \ldots, K\}$) of $\Delta_1^+$ and $\Delta_s^+$  by first computing the $K$ largest eigenvalues $\hat \lambda_l$ ($l\in \{n-K+1, \ldots, n\}$) of $\Delta_1$ and $\Delta_s$ respectively instead and then setting $\lambda_k = 2 - \hat \lambda_{n-k+1}$. There is not a similar property for $\Delta_0^+$ however. Proving upper bounds on the largest eigenvalues of $\Delta_0$ and $\Delta_0^+$ is an active area of research. \cite{guo2005, shu2002, zhang2009signless}.

We use the MATLAB \texttt{eigs} function to calculate the $K$ eigenpairs of the signless Laplacian. This function \cite{lehoucq1996} uses the Implicitly Restarted Arnoldi Method (IRAM) \cite{sorensen1997}, which can efficiently compute the largest eigenvalues and corresponding eigenvectors of sparse matrices. The function \texttt{eigs} firstly computes the orthogonal projection of the matrix you want eigenpairs from, and a random vector, onto the matrix's $K$-dimensional Krylov subspace. This projection is represented by a smaller $K \times K$ matrix. Then \texttt{eigs} calculates the eigenvalues of this $K \times K$ matrix, whose eigenvalues are called Ritz eigenvalues. The Ritz eigenvalues are computed efficiently using a QR method \cite{francis1961}. Computationally these Ritz eigenvalues typically approximate the largest eigenvalues of the original matrix. The time complexity of IRAM is currently unknown, but in practice it produces approximate eigenpairs efficiently.

If the matrix of which the eigenvalues are to be computed is symmetric, the MATLAB \texttt{eigs} function simplifies to the Implicitly Restarted Lanczos Method (IRLM) \cite{calvetti1994}, therefore typically in practice \texttt{eigs} will usually compute the eigenvalues and eigenfunctions of $\Delta_s^+$ faster than those of $\Delta_1^+$.

Using the IRLM for computing the eigenpairs of $\Delta_0^+$ corresponding to its smallest eigenvalues is inefficient. In our experiments using the MATLAB \texttt{eig} function to calculate all eigenpairs of $\Delta_0^+$ and choosing the $K$ eigenpairs corresponding to the smallest eigenvalues for the decomposition \eqref{eq:decomposition} was faster than using the IRLM to calculate the $K$ eigenpairs of $\Delta_0^+$. Hence, the results discussed in this paper are obtained with $\texttt{eig}$ when using $\Delta_0^+$ and \texttt{eigs} when using $\Delta_1^+$ or $\Delta_s^+$.

If we use the MATLAB \texttt{eigs} function when using our spectral decomposition method we cannot a priori determine the time complexity for \ref{alg:signlessMBO}, because practical experiments have shown the complexity of the IRAM and IRLM methods is heavily dependent on the matrix to which they are applied \cite{radke1996}. If we choose to use the MATLAB \texttt{eig} function then the time complexity of \ref{alg:signlessMBO} is $\mathcal{O}(n^3)$, which is the time complexity of computing all eigenpairs of an $n \times n$ matrix. All other remaining steps of \ref{alg:signlessMBO} require fewer operations to compute.


\subsection{Explicit Euler method}

We also compute the solution of \eqref{eq:signlessdiffusion} for $t \in [0,\tau]$ using an explicit finite difference scheme,
\begin{equation}\label{eq:euler}
\begin{cases}
u^{m+1} = u^{m} - \Delta^+u^{m}dt, &\text{ for } m\in \{0, 1, \ldots, M\},\\
u^0 = u(0)
\end{cases}
\end{equation}
for the same choice of $\tau$ as in \eqref{eq:decomposition}. For $M \in \mathbb{N}$, $dt = \frac{\tau}{M}$, and we set $u^*(\tau)= u^M$.

If $G \in \mathcal{G}$ then \ref{alg:signlessMBO} using the Euler method will have a time complexity of $\mathcal{O}(|E|)$, because of the sparsity of the signless Laplacian matrix. When zero entries are ignored, the multiplication of the vector $u^m$ by $\Delta^+$ takes $4|E| + 2n$ operations to compute. Since $G\in \mathcal{G}$ has no isolated nodes, $|E| \geq n-1$, therefore, when $n$ is large enough, $4|E| > 2n$ and hence the time complexity of the multiplication is $\mathcal{O}(|E|)$. All other remaining steps in \ref{alg:signlessMBO} using the Euler method require fewer operations to compute.

In Section~\ref{sec:implicit} we show some results for \ref{alg:signlessMBO} when solving \eqref{eq:signlessdiffusion} using an implicit finite difference scheme, comparing against the results of \ref{alg:signlessMBO} obtained using \eqref{eq:euler} to solve \eqref{eq:signlessdiffusion}.


\subsection{\ref{alg:signlessMBO} pinning condition}\label{sec:pinning condition}

For \ref{alg:signlessMBO} we have that choosing $\tau$ too small causes trivial dynamics in the sense that, for any $j$, $\mu^j=\mu^0$ in \ref{alg:signlessMBO}. In this section we prove a result which shows that such a $\tau$ is inversely proportional to the largest eigenvalue of the signless Laplacian chosen for \ref{alg:signlessMBO}.

We define $d_- := \underset{i \in V}{\mathrm{min}} \: d_i$, and $d_+ := \underset{i \in V}{\mathrm{max}} \: d_i$. Let $\Delta^+ \in \{\Delta_0^+,\Delta_1^+,\Delta_s^+\}$, then the operator norm $\|\Delta^+\|_{\mathcal{V}}$ is defined by
\[
\|\Delta^+\|_{\mathcal{V}} := \underset{u \in \mathcal{V} \setminus \{0\}}{\mathrm{sup}}\frac{\|\Delta^+u\|_{\mathcal{V}}}{\|u\|_{\mathcal{V}}}
\]

We define the maximum norm of $\mathcal{V}$ by $\|u\|_{\mathcal{V},\infty} := \mathrm{max}\{|u_i|: i \in V\}$.

\begin{lemma}
Let $\Delta^+ \in \{\Delta_0^+,\Delta_1^+,\Delta_s^+\}$. The operator norm $\|\Delta^+\|_{\mathcal{V}}$ and the largest eigenvalue $\lambda_n$ of $\Delta^+$ are equal. This implies that, for all $u \in \mathcal{V}$,
\[
\|\Delta^+u\|_{\mathcal{V}} \leq \lambda_n\|u\|_{\mathcal{V}}.
\]
\end{lemma}
\begin{proof}
See \cite[Lemma 2.5]{van2014}.
\end{proof}

\begin{lemma}\label{lem:normequiv}
The norms $\|\cdot\|_{\mathcal{V}}$ and $\|\cdot\|_{\mathcal{V},\infty}$ are equivalent, with optimal constants given by
\[
d_{-}^{\frac{r}{2}}\|u\|_{\mathcal{V},\infty} \leq \|u\|_{\mathcal{V}} \leq \|\chi_V\|_{\mathcal{V}} \|u\|_{\mathcal{V},\infty}.
\]
\end{lemma}
\begin{proof}
See \cite[Lemma 2.2]{van2014}.
\end{proof}

\begin{theorem}\label{thm:pin}
Let $G \in \mathcal{G}$, and let $\lambda_n$ be the largest eigenvalue of the signless Laplacian $\Delta^+ \in \{\Delta_0^+,\Delta_1^+,\Delta_s^+\}$. Let $S_0\subset V$, $\mu^0:= \chi_{S_0}-\chi_{S_0^c}$, and let $\mu^1\in \mathcal{V}^b$ be the result of applying one \ref{alg:signlessMBO} iteration to $\mu^0$. If
\begin{equation}\label{eq:pin}
\tau < \lambda_n^{-1}\mathrm{log}(1 + d_{-}^{\frac{r}{2}}\|\chi_V\|_{\mathcal{V}}^{-1}),
\end{equation} 
then $\mu^1=\mu^0$.
\end{theorem}
\begin{proof}
This proof closely follows the proof of a similar result in \cite[Section 4.2]{van2014}.

If $\|e^{-\tau\Delta^+}\mu^0 - \mu^0\|_{\mathcal{V},\infty} < 1$, then $\mu^1=\mu^0$. Using Lemma~\ref{lem:normequiv}, we compute
\[
\|e^{-\tau\Delta^+}\mu^0 - \mu^0\|_{\mathcal{V},\infty} \leq d_{-}^{-\frac{r}{2}} \|e^{-\tau\Delta^+}\mu^0 - \mu^0\|_{\mathcal{V}} 
\leq d_{-}^{-\frac{r}{2}} \|e^{-\tau\Delta^+} - \mathrm{Id} \|_{\mathcal{V}} \, \|\mu^0\|_{\mathcal{V}} .
\]
Moreover, since $\langle \chi_{S_0}, \chi_{S_0^c}\rangle_\mathcal{V} = 0$, we have 
$
\|\mu^0\|_{\mathcal{V}}^2 = \|\chi_{S_0}\|_{\mathcal{V}}^2 + \|\chi_{S_0^c}\|_{\mathcal{V}}^2 = \|\chi_{S_0}+\chi_{S_0^c}\|_{\mathcal{V}}^2 = \|\chi_V\|_{\mathcal{V}}^2.
$

Using the triangle inequality and the submultiplicative property (see \cite{rudin1964} for example) of $\|\cdot\|_{\mathcal{V}}$, we compute
$
\|e^{-\tau \Delta^+} - \mathrm{Id}\|_{\mathcal{V}} \leq \sum_{k=1}^{\infty}\frac{1}{k!}(\tau\|\Delta^+\|_{\mathcal{V}})^k = e^{\lambda_n\tau} - 1.
$ 
Therefore, if $\tau < \lambda_n^{-1}\mathrm{log}(1 + d_{-}^{\frac{r}{2}}\|\chi_V\|^{-1})$, then $\mu^1=\mu^0$.
\end{proof}

As stated in Section~\ref{sec:spectral}, we choose $\tau = 20$ as diffusion time for \ref{alg:signlessMBO} using $\Delta_1^+$ or $\Delta_s^+$, and $\tau = \frac{40}{\lambda_n}$ when using \ref{alg:signlessMBO} with $\Delta_0^+$ as the choice of operator. This is due to $\tau = 20$ often being too large when using \ref{alg:signlessMBO} with $\Delta_0^+$. Choosing $\tau = 20$ for \ref{alg:signlessMBO} using $\Delta_0^+$ causes the solution to converge to $u(\tau) = 0$ to machine precision. We therefore choose $\tau = \frac{40}{\lambda_n}$ for $\Delta_0^+$ since \ref{thm:pin} implies a suitable choice of $\tau$ for \ref{alg:signlessMBO} with respect to obtaining non-trivial output cuts is inversely proportional to the largest eigenvalue of the chosen operator $\Delta^+$. Since $\lambda_n = 2$ for $\Delta_1^+$ and $\Delta_s^+$ we choose to divide $\tau$ by $\frac{\lambda_n}{2}$ for $\Delta_0^+$.


\section{Results}\label{sec:results}

\subsection{Method}\label{sec:method}

\begin{figure}
\centering
\begin{subfigure}{.5\textwidth}
  \centering
  \includegraphics[width=7cm]{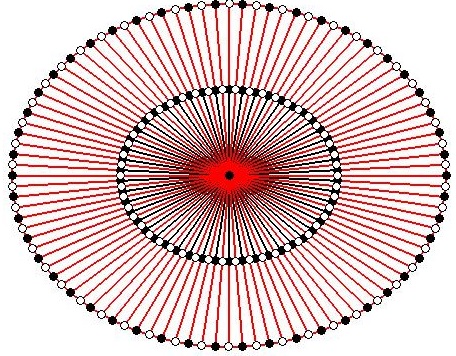}
  \caption{Web graph, maximum cut approximation}\label{fig:webgraph}
\end{subfigure}%
\begin{subfigure}{.5\textwidth}
  \centering
  \includegraphics[width=7cm]{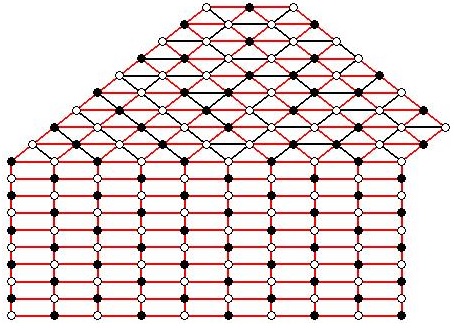}
  \caption{Square-triangle mesh maximum cut approximation}\label{fig:squaretriangle}
\end{subfigure}
\caption{Visualisation of maximum cut approximations (best viewed in colour)}\label{fig:visualisation}
\end{figure}

In Section~\ref{sec:results} we compare the results of our new algorithm \ref{alg:signlessMBO} with the results obtained by \ref{alg:GW}. In Sections~\ref{sec:Rand}--\ref{sec:large} we display the results of \ref{alg:signlessMBO} using both the spectral decomposition method and the explicit Euler method, fixing the variable $\tau$ for both methods. We run all our tests on a Windows 7 PC with 16GB RAM and an Intel(R) Core(TM) i5-4590 CPU with clock speed 3.30GHz. For both \ref{alg:signlessMBO} and \ref{alg:GW} we use MATLAB, which is convenient to use when dealing with large sparse matrices.

For all of our tests using the spectral decomposition method we choose $K = \floor{\frac{n}{100}}$. In practice it reduces the computation time without sacrificing much accuracy in the cut approximations. We further analyse this choice in Section~\ref{sec:parameters}.  For all of our tests using the Euler method we set $M=100$, in order to keep $dt$ small so as to ensure stability on our explicit scheme. We compute the \ref{alg:signlessMBO} evolutions for 50 initial conditions chosen at random from $\mathcal{V}^b$. In the tables which we refer to in this section, we state the greatest (Best), average (Avg), and smallest (Least) sizes of cuts obtained by these 50 runs of \ref{alg:signlessMBO}. We run \ref{alg:signlessMBO} using $\Delta_0^+$, $\Delta_1^+$ and $\Delta_s^+$, fixing the initial conditions for each operator, using both the spectral method and the Euler method for each operator, and compare the results. 

We compare the results of \ref{alg:signlessMBO} with those of \ref{alg:GW}. To compute the relaxation step of \ref{alg:GW} we use SDPT3 MATLAB software \cite{tutuncu2003} as it exploits the sparse structure of the matrices we work on. According to \cite{mittelmann2012} it is best suited for both smaller problems and for larger problems with sparse matrices. The stopping tolerance is set as $|Z_P^* - Z_D^*| < 10^{-6}$. The recommended tolerance for the SDPT3 software is set as $10^{-8}$. However, in our experiments increasing this tolerance to $10^{-6}$ reduced the computation time of \ref{alg:GW}, without any change in output cut sizes.  After the relaxation step, we perform the hyperplane step 50 times, randomly choosing a vector $r$ each time. Each choice of $r$ leads to a resulting cut; in the tables referred to in this section, we list the highest (Best), average (Avg), and lowest (Least) sizes of these cuts. In each of these categories in our tables we highlight the method that obtained the best result, \ref{alg:signlessMBO} using $\Delta_0^+$, \ref{alg:signlessMBO} using $\Delta_1^+$, \ref{alg:signlessMBO} using $\Delta_s^+$, or \ref{alg:GW}. We do the same for the run times (Time) of each method.

For both \ref{alg:signlessMBO} and \ref{alg:GW} only the adjacency matrix and the parameter choice $\eta$ is initially provided, therefore the reported run times cover all calculations from that starting point. For each graph we remove the isolated nodes by removing all rows and columns of the graph's adjacency matrix which have all zero entries. (This does not affect the size of any cut of the graph.) For the spectral decompostion variant of \ref{alg:signlessMBO} using $\Delta_1^+$ and $\Delta_s^+$ this includes removing all isolated nodes, computing the matrices $L_1$ and $L_s$, finding their $K$ eigenpairs corresponding to the leading eigenvalues in order to use Proposition~\ref{prop:eigenpairs}, to compute the eigenpairs corresponding to the trailing eigenvalues of $L_1^+$ and $L_s^+$ respectively, generating initial conditions, running the signless diffusion and thresholding steps, and computing the size of the cut from each MBO iteration. For $\Delta_s^+$ the computation time includes calculating $L_1$ in order to compute the size of the output cuts using \eqref{eq:cutLaplacian}. The computation time for \ref{alg:signlessMBO} using $\Delta_0^+$ includes removing all isolated nodes, computing the matrix $L_0^+$, finding all its eigenpairs, choosing the largest eigenvalue for the time step $\tau$, and using the $K$ eigenpairs corresponding to the smallest eigenvalues for the remaining steps.

For the explicit Euler method variant of \ref{alg:signlessMBO} the computation time includes removing all isolated nodes, computing $L^+ \in \{L_0^+,L_1^+,L_s^+\}$, generating initial conditions, running the signless diffusion and thresholding steps, and computing the size of the cut induced by each MBO iteration. For $L = L_s$ we also compute $L_1$ to obtain the size of the output cut using \eqref{eq:cutLaplacian}.

For every graph there exists a $\tau_{max}$ such that for all $\tau \geq \tau_{max}$ the solution to \eqref{eq:signlessdiffusion} computed using \ref{alg:signlessMBO} converges to $u(\tau) = 0$ to machine precision. In practice $\tau_{max}$ is dependent on the operator $\Delta^+$. In our experiments we see that choosing a $\tau$ which is in between the pinning condition in Theorem~\ref{thm:pin} and $\tau_{max}$ is difficult due to the difference between them being small when $\Delta_0^+$ is our operator for \ref{alg:signlessMBO}. In Section~\ref{sec:scale} and Section~\ref{sec:weighted} we run our experiments on graphs with a scale free structure (see Section~\ref{sec:scale}). When running \ref{alg:signlessMBO} using the explicit Euler method and $\Delta_0^+$ we encounter problems in choosing suitable $\tau$ and $dt$ for such graphs. This is due to the inflexibility of choosing $\tau$ such that it is less than $\tau_{max}$ and also greater than the bound in Theorem~\ref{thm:pin}. Since the Euler method is an approximation of the spectral method, we encounter problems in this case. If \ref{alg:signlessMBO} returns a cut which has pinned due to Theorem~\ref{thm:pin} or is zero due to the solution of \eqref{eq:signlessdiffusion} converging to zero to machine precision then we refer to the cut as a trivial cut. In Section~\ref{sec:implicit} we show that it is possible to obtain non-trivial cut sizes using \ref{alg:signlessMBO} with $\Delta_0^+$ by solving \eqref{eq:signlessdiffusion} using an implicit Euler scheme.        

Figure~\ref{fig:visualisation} shows two examples of approximate maximum cuts obtained with the \ref{alg:signlessMBO} algorithm. The black nodes are in $V_1$ and the white nodes are in $V_{-1}$. An edge is coloured red, if it connects two nodes of different colour, i.e. if it contributes to the size of the cut. If it does not, it is black.

Figure~\ref{fig:webgraph} shows an unweighted web graph which has 201 nodes and 400 edges. We set $\tau = 20$ in \ref{alg:signlessMBO} using $\Delta_1^+$ and the Euler method to solve \eqref{eq:signlessdiffusion}. The resulting approximation of the maximum cut value is 350. The run time is 0.09 seconds. Figure~\ref{fig:squaretriangle} shows an unweighted triangle-square graph which has 162 nodes and 355 edges. We set $\tau = 20$ and $K=20$ in \ref{alg:signlessMBO} using $\Delta_1^+$ and the spectral method to solve \eqref{eq:signlessdiffusion}. The approximation of the maximum cut value is 295 and the run time is 0.14 seconds.


\subsection{Random graphs}\label{sec:Rand}

In Figures~\ref{fig:G1000}, \ref{fig:G2500}, and~\ref{fig:G5000} we list results obtained for Erd\"os-R\'enyi graphs. 

For each of $G(1000,0.01)$ (Figure~\ref{fig:G1000}), $G(2500,0.4)$ (Figure~\ref{fig:G2500}), and $G(5000,0.001)$ (Figure~\ref{fig:G5000}) we create 100 realisations. We then run \ref{alg:signlessMBO} with both the spectral method and the Euler method, and we run \ref{alg:GW}. For both of the \ref{alg:signlessMBO} methods we choose either $\Delta_0^+$, $\Delta_1^+$, or $\Delta_s^+$, setting $\tau = 20$ for all tests. The bar chart represents the mean of the best, average, and least cuts over all 100 realisations of the chosen random graph. The error bars are the corrected sample standard deviation\footnote{The corrected sample standard deviation is computed using MATLAB's \texttt{std} code in all experiments in this paper.} of the results obtained over all 100 realisations. Figure~\ref{fig:G1000} shows that \ref{alg:signlessMBO} using either the spectral method or Euler method for $\Delta_1^+$ and $\Delta_s^+$ produces better mean best, mean average, and mean least cuts than \ref{alg:GW} on this set of graphs. Figure~\ref{fig:G2500} shows that \ref{alg:signlessMBO} using the spectral method and either $\Delta_1^+$ or $\Delta_s^+$ produces better mean cut approximations than \ref{alg:GW} on this set of graphs. Figure~\ref{fig:G5000} shows the same conclusions as Figure~\ref{fig:G1000} for this set of graphs. Table~\ref{tab:ERTime} shows that \ref{alg:signlessMBO} using the spectral method produces the fastest run times on all three types of Erd\"os-R\'enyi graphs that we test on. We note that \ref{alg:GW} has a superior run time over \ref{alg:signlessMBO} using the Euler method on the realisations of $G(2500,0.4)$.

\begin{figure}
\centering
  \includegraphics[width=12cm]{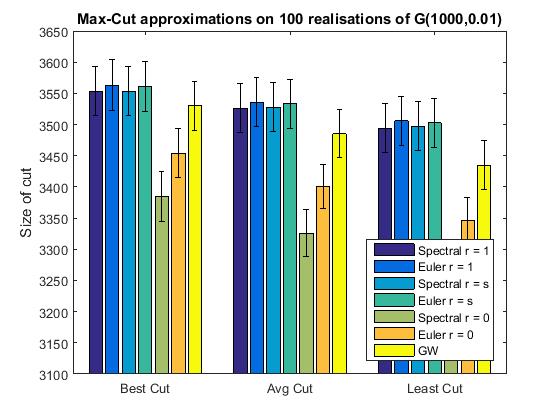}
\caption{Bar chart of Max-Cut approximations on 100 realisations of $G(1000,0.01)$.}\label{fig:G1000}
\end{figure}

\begin{figure}
\centering
  \includegraphics[width=12cm]{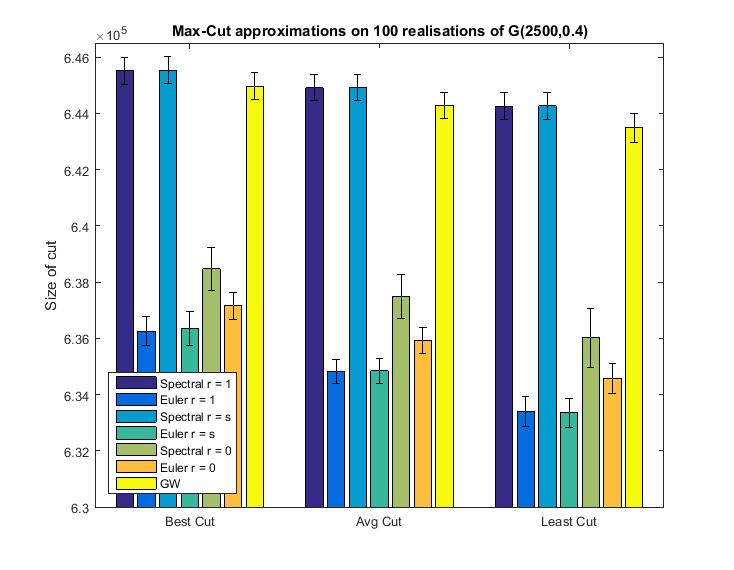}
\caption{Bar chart of Max-Cut approximations on 100 realisations of $G(2500,0.4)$.}\label{fig:G2500}
\end{figure}

\begin{figure}
\centering
  \includegraphics[width=12cm]{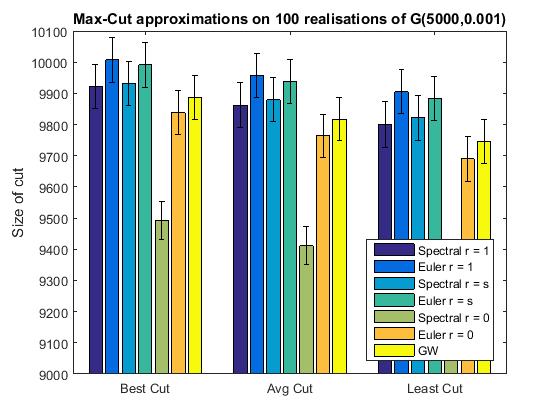}
\caption{Bar chart of Max-Cut approximations on 100 realisations of $G(5000,0.001)$.}\label{fig:G5000}
\end{figure}

\begin{table}
\begin{tabular}{ |l|l|l|l|l|l|l|l|}
\hline
Graph & $\Delta_1^+$ (S) & $\Delta_1^+$ (E) & $\Delta_s^+$ (S) & $\Delta_s^+$ (E) & $\Delta_0^+$ (S) & $\Delta_0^+$ (E) & GW \\ \hline
$G(1000,0.01)$ & \textbf{0.20} & 1.58 & 0.34 & 1.52 & 0.56 & 1.06 & 5.25 \\ \hline
$G(2500,0.4)$ & 8.04 & 172.91 & 13.33 & 181.40 & \textbf{6.40} & 172.73 & 55.36 \\ \hline
$G(5000,0.001)$ & \textbf{4.38} & 16.96 & 6.37 & 14.95 & 24.99 & 6.97 & 257.09\\ \hline
\end{tabular}
\caption{Average \ref{alg:signlessMBO} and \ref{alg:GW} run-times for each realisation of $G(n,p)$, time in seconds.}\label{tab:ERTime}
\end{table}


\subsection{Scale-free graphs}\label{sec:scale}

The degree distribution $P: \mathbb{N} \rightarrow \mathbb{R}$ of an unweighted graph $G$ is given by $P(j) := \frac{|\{i \in V: \: d_i = j\}|}{n}$. Random graphs such as the ones discussed in Section~\ref{sec:Rand} have a degree distribution which resembles a normal distribution. The graph $G \in \mathcal{G}$ is a scale-free graph if its degree distribution roughly follows a power law, i.e $P(j) \approx j^{-\gamma}$, where often in practice, $\gamma \in (2,3)$ \cite{barabasi2009}. Scale-free graphs have become of interest as graphs such as internet networks, collaboration networks, and social networks are conjectured to more closely resemble scale-free graphs instead of random graphs \cite{barabasi2003}.

\begin{figure}
\centering
\begin{subfigure}{.5\textwidth}
  \centering
  \includegraphics[width=7cm]{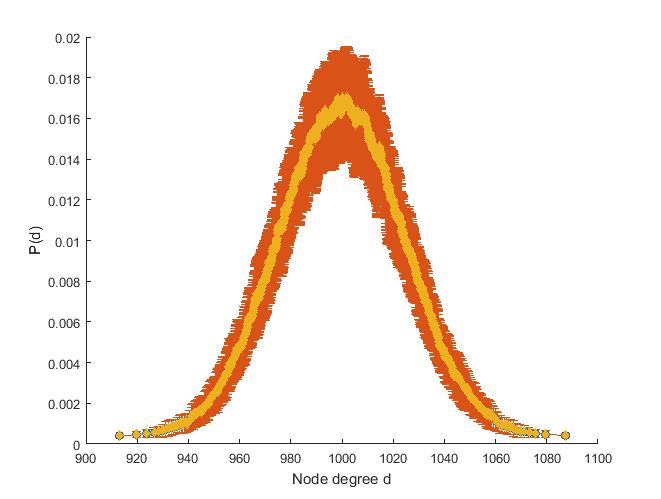}
  \caption{Degree distribution of a realisation of $G(2500,0.4)$.}\label{fig:g2500deg}
\end{subfigure}%
\begin{subfigure}{.5\textwidth}
  \centering
  \includegraphics[width=7cm]{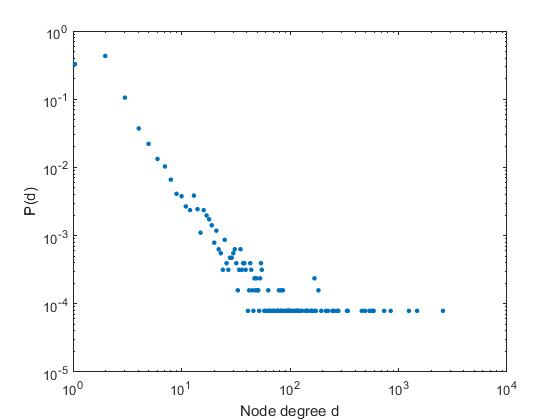}
  \caption{Degree distribution of the AS1 Graph.}\label{fig:AS1deg}
\end{subfigure}
\caption{Average degree distribution of 100 realisations of a random graph and the degree distribution of a scale free graph.}\label{fig:DEG}
\end{figure}

In Table~\ref{table:large} we list results for some scale free graphs. We test the algorithms on 8 autonomous systems internet graphs, $\mathrm{AS}i$, $i\in \{1, \ldots, 8\}$. These graphs represent smaller imitations of an internet network, which were acquired from the website \cite{ASI}. We also test on the graph Gnutella09 which is a model of a peer to peer file sharing network, and the graph WikiVote, which is a network representing a Wikipedia administrator election, both obtained from \cite{snapnets}. All of the scale free graphs in this section are unweighted and undirected graphs.

Table~\ref{table:ASProp} displays some properties of the random graphs in Section~\ref{sec:Rand} and the scale-free graphs we test on. Figure~\ref{fig:DEG} displays the average degree distribution of 100 realisations of $G(2500,0.4)$, in Figure~\ref{fig:g2500deg}, and the degree distribution of the AS1 Graph, in Figure~\ref{fig:AS1deg}. In Figure~\ref{fig:g2500deg} the yellow points indicate the degree distribution, and the orange lines indicate the corrected sample standard deviation of the average degree distribution. In Figure~\ref{fig:AS1deg} the blue dots indicate the degree distribution. As we see, the average degree distribution of the realisations of $G(2500,0.4)$ is similar to a normal distribution, and the degree distribution of the AS1 graph resembles a power law, as expected.

\begin{table}
\centering
\begin{tabular}{ |l|l|l|l|l|}
\hline
Graph & $|V|$ & $|E|$ & $d_{-}$ & $d_{+}$ \\ \hline
$G(1000,0.01)$(1) & 1000 & 4919 & 1 & 21 \\ \hline
$G(1000,0.01)$(2) & 1000 & 4939  & 2 & 21 \\ \hline
$G(2500,0.4)$(1) & 2500 & 1248937 & 910 & 1079  \\ \hline
$G(2500,0.4)$(2) & 2500 & 1251182 & 904 & 1081\\ \hline
$G(5000,0.001)$(1) & 4962 & 12646 & 1 & 16 \\ \hline
$G(5000,0.001)$(2) & 4969 & 12642 & 1 & 16 \\ \hline
\end{tabular}
\
\begin{tabular}{ |l|l|l|l|l|}
\hline
Graph & $|V|$ & $|E|$ & $d_{-}$ & $d_{+}$ \\ \hline
AS1 & 12694 & 26559 & 1 & 2566  \\ \hline
AS2 & 7690 & 15413 & 1 & 1713  \\ \hline
AS3 & 8689 & 17709 & 1 & 1911  \\ \hline
AS4 & 8904 & 17653 & 1 & 1921  \\ \hline
GNutella09 & 8114 & 26013 & 1 & 102 \\ \hline
Wiki-Vote & 7115 & 100762 & 1 & 1065 \\ \hline
\end{tabular}
\caption{Properties of $G(n,p)$ graph realisations vs scale free graphs.}\label{table:ASProp}
\end{table}

For all graphs listed in Table~\ref{table:large}, using either $\Delta_1^+$ or $\Delta_s^+$ \ref{alg:signlessMBO} using the Euler method or the spectral method outperforms \ref{alg:GW} with respect to the average and least obtained cut sizes and the run time, but \ref{alg:GW} obtains the best results when considering the greatest obtained cuts. For any choice of $\Delta_1^+$ and $\Delta_s^+$ and for any choice of signless diffusion solver the greatest cuts obtained by \ref{alg:signlessMBO} are all at least 98.1\% of the greatest cut size obtained by \ref{alg:GW}. The difference in run times is notable though. The time taken by \ref{alg:signlessMBO} stays below 30 seconds for all graphs in Table~\ref{table:large}, irrespective of choice of Laplacian and signless diffusion solver. However, the \ref{alg:GW} algorithm's run times range between 9 and 44 minutes. These results suggest that \ref{alg:signlessMBO} using $\Delta_1^+$ or $\Delta_s^+$, and using either signless diffusion solver offers a significant decrease in run time at the cost of about 1-2\% accuracy in the resulting cut size, in comparison with \ref{alg:GW}, when applied to the graphs in Table~\ref{table:large}.

\begin{table}
\centering

\begin{tabular}{ |l|l|l|l|l|}
\hline
Graph & $\Delta_1^+$ (S) Best & $\Delta_1^+$ (S) Avg & $\Delta_1^+$ (S) Least & $\Delta_1^+$ (S) Time \\ \hline
AS1        & 22744 & 22542.20  & 22183 & \textbf{15.85} \\ \hline
AS2        & 13249 & 13153.72 & 13054 & \textbf{3.55}  \\ \hline
AS3        & 15118 & 15027.22 & 14907 & 4.73  \\ \hline
AS4        & 15194 & 15143.44 & 15042 & \textbf{5.67}  \\ \hline
AS5        & 14080 & 13988.90  & 13928 & \textbf{4.82}  \\ \hline
AS6        & 18053 & 17964.74 & 17876 & 10.06 \\ \hline
AS7        & 22741 & 22535.00    & 22150 & 17.82 \\ \hline
AS8        & 22990 & 22720.36 & 22334 & 17.22 \\ \hline
GNutella09 & 20280 & 20143.74 & 19983 & 8.16  \\ \hline
WikiVote   & 72981 & 72856.40  & 72744 & 2.46  \\ \hline
\end{tabular}
\vspace{0.5cm}

\begin{tabular}{ |l|l|l|l|l|}
\hline
Graph & $\Delta_1^+$ (E) Best & $\Delta_1^+$ (E) Avg & $\Delta_1^+$ (E) Least & $\Delta_1^+$ (E) Time \\ \hline
AS1        & 22798 & \textbf{22670.76} & 22268 & 23.62 \\ \hline
AS2        & 13281 & \textbf{13199.72} & \textbf{13120} & 8.76  \\ \hline
AS3        & 15175 & \textbf{15095.46} & \textbf{15007} & 9.95  \\ \hline
AS4        & 15270 & \textbf{15202.70}  & \textbf{15117} & 10.88 \\ \hline
AS5        & 14120 & \textbf{14020.62} & \textbf{13944} & 9.50   \\ \hline
AS6        & 18134 & \textbf{18034.10}  & \textbf{17933} & 16.50  \\ \hline
AS7        & 22826 & \textbf{22696.42} & \textbf{22525} & 25.78 \\ \hline
AS8        & 23070 & \textbf{22951.54} & \textbf{22550} & 25.38 \\ \hline
GNutella09 & 20437 & \textbf{20361.92} & \textbf{20295} & 17.14 \\ \hline
WikiVote   & 73159 & \textbf{73126.34} & \textbf{73086} & 9.06  \\ \hline
\end{tabular}
\caption{\ref{alg:signlessMBO} cut approximations using $\Delta_1^+$ on graphs with a scale free structure, time in seconds.}
\end{table}

\vspace{0.5cm}

\begin{table}
\centering
\begin{tabular}{ |l|l|l|l|l|}
\hline
Graph & $\Delta_s^+$ (S) Best & $\Delta_s^+$ (S) Avg & $\Delta_s^+$ (S) Least & $\Delta_s^+$ (S) Time \\ \hline
AS1        & 22809 & 22620.8  & \textbf{22325} & 17.83 \\ \hline
AS2        & 13271 & 13178.86 & 13103 & 4.12  \\ \hline
AS3        & 15166 & 15082.1  & 14992 & \textbf{4.66}  \\ \hline
AS4        & 15237 & 15166.24 & 15077 & 5.78  \\ \hline
AS5        & 14075 & 14011.96 & 13911 & 5.47  \\ \hline
AS6        & 18088 & 17968.04 & 17859 & \textbf{9.14}  \\ \hline
AS7        & 22822 & 22629.66 & 22218 & \textbf{15.73} \\ \hline
AS8        & 23061 & 22884.8  & 22547 & \textbf{15.46} \\ \hline
GNutella09 & 20282 & 20186.32 & 20101 & \textbf{6.82}  \\ \hline
WikiVote   & 73169 & 73003.44 & 72917 & \textbf{2.25} \\ \hline
\end{tabular}

\vspace{0.5cm}

\begin{tabular}{ |l|l|l|l|l|}
\hline
Graph & $\Delta_s^+$ (E) Best & $\Delta_s^+$ (E) Avg & $\Delta_s^+$ (E) Least & $\Delta_s^+$ (E) Time \\ \hline
AS1        & 22789 & 22629.62 & 22261 & 27.63 \\ \hline
AS2        & 13256 & 13176.64 & 13094 & 9.09  \\ \hline
AS3        & 15139 & 15059.54 & 14967 & 10.24 \\ \hline
AS4        & 15234 & 15159.76 & 15079 & 11.57 \\ \hline
AS5        & 14096 & 14011.9  & 13930 & 10.47 \\ \hline
AS6        & 18088 & 17994.66 & 17876 & 16.12 \\ \hline
AS7        & 22823 & 22639.58 & 22237 & 24.5  \\ \hline
AS8        & 23036 & 22865    & 22440 & 25.08 \\ \hline
GNutella09 & 20397 & 20332.28 & 20170 & 18.75 \\ \hline
WikiVote   & 72993 & 72772.26 & 72549 & 9.00 \\  \hline 
\end{tabular}
\caption{\ref{alg:signlessMBO} cut approximations using $\Delta_s^+$ on graphs with a scale free structure, time in seconds.}
\end{table}

\vspace{0.5cm}

\begin{table}
\centering
\begin{tabular}{ |l|l|l|l|l|}
\hline
Graph & $\Delta_0^+$ (S) Best & $\Delta_0^+$ (S) Avg & $\Delta_0^+$ (S) Least & $\Delta_0^+$ (S) Time \\ \hline
AS1  &  22578  &  22303.10  & 21844  & 297.79  \\ \hline
AS2  &  13081 & 12935.80  & 12763  & 62.41  \\ \hline
AS3  & 14995 & 14869.52 & 14702 & 90.32   \\ \hline
AS4  & 15097  & 14994.92  &  14885 & 88.53   \\ \hline
AS5  &  13952  &  13795.24   &  13561    &  70.81 \\ \hline
AS6  &  17836  & 17672.50  & 17527  & 149.60 \\ \hline
AS7  & 22571  & 22328.18 & 21932  & 294.26  \\ \hline
AS8  & 22824  & 22585.88  & 22075  & 287.79 \\ \hline
GNutella09  & 19079  & 18419.36  & 17951   & 72.03 \\ \hline
WikiVote & 65504  & 60599.74   & 56917  & 46.11 \\ \hline
\end{tabular}
\caption{\ref{alg:signlessMBO} cut approximations using $\Delta_0^+$ on graphs with a scale free structure, time in seconds.}
\end{table}

\vspace{0.5cm}

\begin{table}
\centering
\begin{tabular}{ |l|l|l|l|l|}
\hline
Graph & GW Best & GW Avg & GW Least & GW Time\\ \hline
AS1 & \textbf{22864} & 22346.26 & 20546 & 2324.98\\ \hline
AS2 & \textbf{13328} & 13039.10 & 12048 & 594.29\\ \hline
AS3 & \textbf{15240} & 14961.56 & 14050 & 826.65\\ \hline
AS4 & \textbf{15328} & 15015.34 & 14072 & 832.28\\ \hline
AS5 & \textbf{14190} & 13810.82 & 12922 & 721.51 \\ \hline
AS6 & \textbf{18191} & 17851.24 & 16483 & 1368.35 \\ \hline
AS7 & \textbf{22901} & 22421.80 & 21244 & 2321.34 \\ \hline
AS8 & \textbf{23170} & 22593.10 & 21110 & 2613.62 \\ \hline
GNutella09 & \textbf{20658} & 20242.02 & 18815 & 1095.04\\ \hline
Wiki-Vote & \textbf{73363} & 71510 & 62886 & 1074.98\\ \hline
\end{tabular}
\caption{\ref{alg:GW} cut approximations on graphs with a scale free structure, time in seconds.}\label{table:large}
\end{table}


\subsection{Random modular graphs}\label{sec:mod}

Modular graphs have a community structure. Nodes in a community have many connections with other members of the same community and noticeably fewer connections with members of other communities. In Figure~\ref{fig:4mod} we show what our Max-Cut approximation looks like on a random modular graph. We generate realisations of random unweighted modular graphs $R(n,c,p,r)$ using the code provided at \cite{Mod}. The variables for the graph are the number of nodes $n$, the number $c \in \mathbb{N}$ of communities that the graph contains, a probability $p$ such that the graph will have an expected number of $\frac{n^2}{2p}$ edges, and a ratio $r \in [0,1]$, with $r|E|$ being the expected number of edges connecting nodes in the same community and $(1-r)|E|$ being the expected number of edges connecting nodes in different communities.

\begin{figure}
\centering
  \includegraphics[width=10cm]{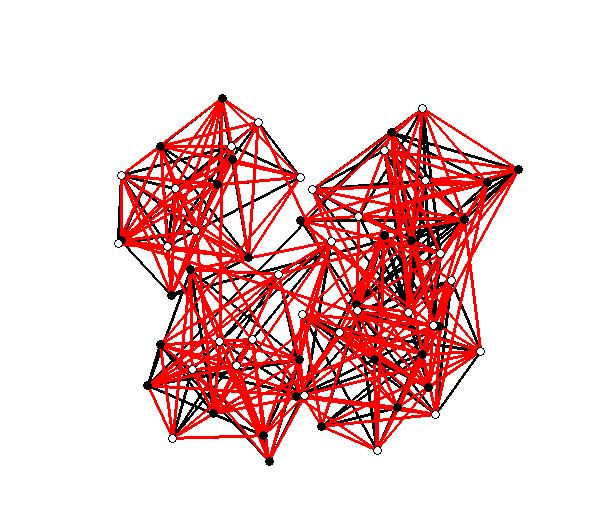}
  \caption{A Max-Cut approximation on a random 4-modular graph (best viewed in colour).}\label{fig:4mod}
\end{figure}

In Figures~\ref{fig:mod2500}, \ref{fig:mod4000}, and~\ref{fig:mod10000} we display results obtained for random modular graphs. For each of $R(2500,2,0.009,0.8)$ (Figure~\ref{fig:mod2500}), $R(4000,20,0.01,0.7)$ (Figure~\ref{fig:mod4000}), and $R(10000,10,0.01,0.8)$ (Figure~\ref{fig:mod10000}) we create 100 realisations. We then run \ref{alg:signlessMBO} with both the spectral method and the Euler method, and we run \ref{alg:GW}. For both of the \ref{alg:signlessMBO} methods we choose either $\Delta_0^+$, $\Delta_1^+$, or $\Delta_s^+$, setting $\tau = 20$ for all tests. The bar chart represents the mean of the best, average, and least cuts over all 100 realisations of the chosen random modular graph. The error bars are the corrected sample standard deviation of the results obtained over all 100 realisations.

In Figures~\ref{fig:mod2500}, \ref{fig:mod4000}, and \ref{fig:mod10000} we see that using either $\Delta_1^+$ or $\Delta_s^+$ \ref{alg:signlessMBO} with both the spectral method and the Euler method outperforms \ref{alg:GW} with respect to the best, average, and least cuts. In Table~\ref{table:ModTime} we see that for any choice of operator and method, \ref{alg:signlessMBO} is faster on average than \ref{alg:GW} for our choices for random modular graphs. We note in particular that for our realisations of $R(10000,10,0.01,0.8)$ the average \ref{alg:GW} test took just below 65 minutes, where as the average \ref{alg:signlessMBO} test using the spectral method and either $\Delta_1^+$ or $\Delta_s^+$ took under a minute, obtaining on average better outcomes.

\begin{figure}
\centering
  \includegraphics[width=12cm]{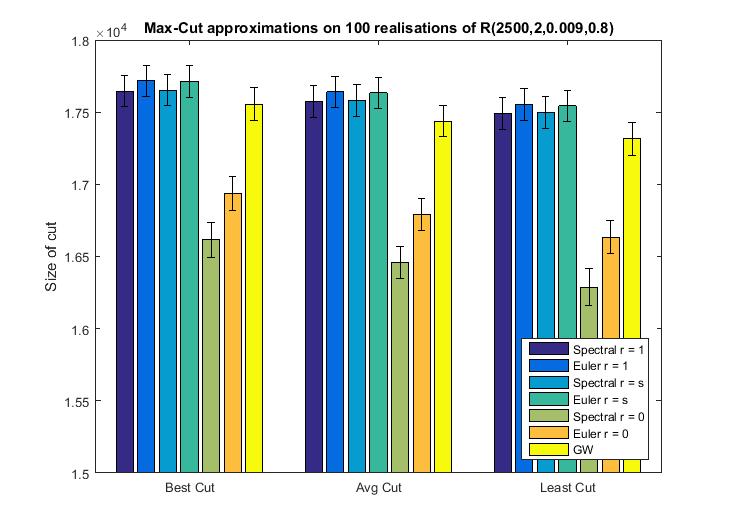}
  \caption{Bar chart of Max-Cut approximations on 100 realisations of $R(2500,2,0.009,0.8)$.}\label{fig:mod2500}
\end{figure}

\begin{figure}
\centering
  \includegraphics[width=12cm]{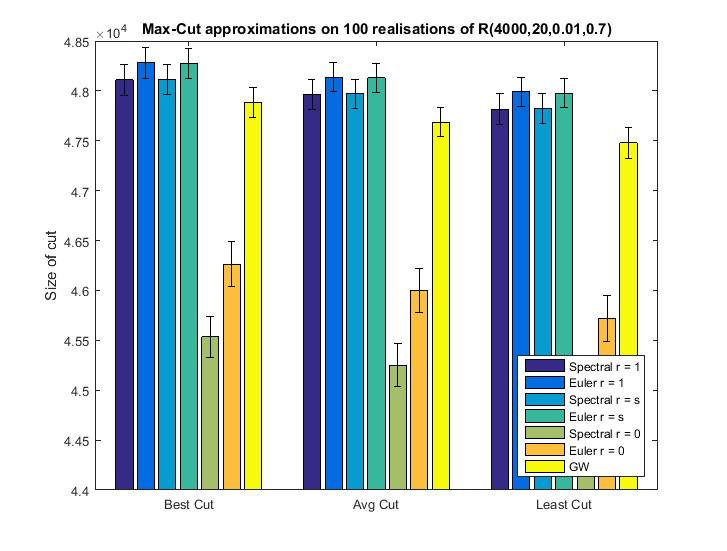}
  \caption{Bar chart of Max-Cut approximations on 100 realisations of $R(4000,20,0.01,0.7)$.}\label{fig:mod4000}
\end{figure}

\begin{figure}
\centering
  \includegraphics[width=12cm]{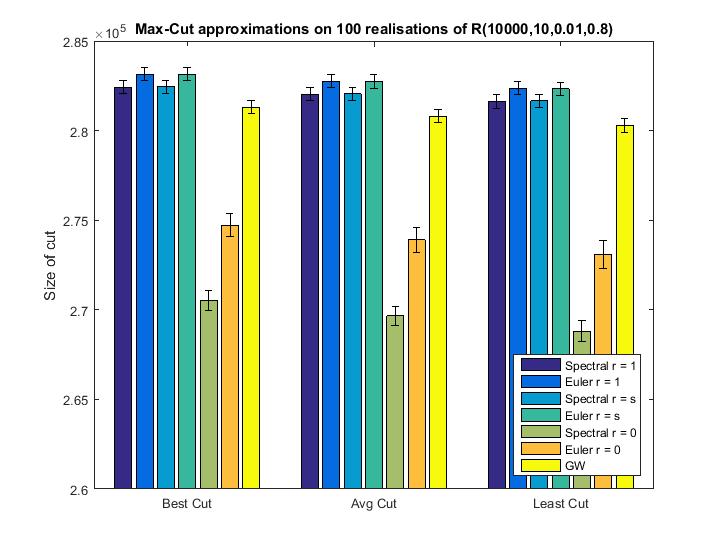}
  \caption{Bar chart of Max-Cut approximations on 100 realisations of $R(10000,10,0.01,0.8)$.}\label{fig:mod10000}
\end{figure}

\begin{table}
\begin{tabular}{ |l|l|l|l|l|l|l|l|}
\hline
Graph & $\Delta_1^+$ (S) & $\Delta_1^+$ (E) & $\Delta_s^+$ (S) & $\Delta_s^+$ (E) & $\Delta_0^+$ (S) & $\Delta_0^+$ (E) & GW \\ \hline
$R(2500,2,0.009,0.8)$ & 0.80 & 10.43 & \textbf{0.79} & 10.26 & 4.36  & 6.13 & 56.30  \\ \hline
$R(4000,20,0.01,0.7)$ & \textbf{4.05} & 30.46 & 4.49 & 29.52 & 16.26 & 18.19  & 248.25  \\ \hline
$R(10000,10,0.01,0.8)$ & \textbf{49.98} & 266.10 & 52.85 & 266.40 & 210.94 & 194.52 & 3893.87 \\ \hline
\end{tabular}
\caption{Average \ref{alg:signlessMBO} and \ref{alg:GW} run-times for each realisation of $R(n,c,p,r)$, time in seconds.}\label{table:ModTime}
\end{table}


\subsection{Weighted graphs}\label{sec:weighted}

In this subsection we assign random weights to the edges of selected graphs from Section~\ref{sec:Rand} and Section~\ref{sec:scale}. To create the graphs W1 and W2 we use two of the realisations of $G(1000,0.01)$, and multiply its edges by random real numbers drawn uniformly from in the range $[0,2]$ and $[0,20]$ respectively. W3 and W4 were created by using two of the realisations of $G(2500,0.4)$ in Section~\ref{sec:Rand}, and multiplying its edges by random real numbers drawn uniformly from in the ranges $[0,5]$ and $[0,1]$ respectively. W5, W6, W7 were created by using three of the realisations of  $G(5000,0.001)$ in Section~\ref{sec:Rand}, and multiplying its edges by random real numbers drawn uniformly from in the ranges  $[0,1],[0,15],$ and $[0,50]$ respectively. W8 is the AS1 graph, whose edges are multiplied by random real numbers drawn uniformly from in the range $[0,12]$, W9 is the AS5 graph whose edges are multiplied by random real numbers drawn uniformly from in the range $[0,4]$ and W10 is the AS8 graph whose edges are multiplied by random real numbers drawn uniformly from in the range $[0,8]$. We run \ref{alg:signlessMBO} for all three choices of $\Delta^+$, on all of these graphs, and compare against \ref{alg:GW} in Table~\ref{table:weighted}. We set $\tau = 20$ for both the spectral decomposition method and the Euler method.

We saw that \ref{alg:signlessMBO} using the spectral method produced larger cuts than \ref{alg:GW} on the random graphs considered in Section~\ref{sec:Rand}; when assigning random weights to the edges of these random graphs the same conclusion holds. We see in Table~\ref{table:weighted} that for this collection of random graphs \ref{alg:signlessMBO} using the spectral method (with either $\Delta_1^+$ or $\Delta_s^+$ used) outperforms \ref{alg:GW} with respect to the best, average, and smallest obtained cut sizes, and the run time. In Section~\ref{sec:scale} we saw that \ref{alg:signlessMBO} using both the spectral method and the Euler method produced better average and smallest cuts than \ref{alg:GW} on the scale free graphs considered in that section, but the best cut sizes were produced more often by \ref{alg:GW}. These weighted examples support the same conclusions. The blank results in Table~\ref{table:weighted} for the Euler method using $\Delta_0^+$ are due to \ref{alg:signlessMBO} producing trivial results for these choices as stated in Section~\ref{sec:method}.

\begin{table}
\centering
\begin{tabular}{ |l|l|l|l|l|}
\hline
Graph & $\Delta_1^+$ (S) Best & $\Delta_1^+$ (S) Avg & $\Delta_1^+$ (S) Least & $\Delta_1^+$ (S) Time \\ \hline
W1 & 3612.00 & 3569.08 & 3537.10 & 0.47 \\ \hline
W2 & 36487.51 & 36082.58 & 35687.87 & \textbf{0.30} \\ \hline
W3 & 1622125.53 & \textbf{1620885.77} & \textbf{1619371.25} & 8.09\\ \hline
W4 & 323926.34 & 323639.05 & \textbf{323321.92} & 8.59 \\ \hline
W5 & 5054.26 & 5033.54 & 5010.38 & \textbf{4.00}\\ \hline
W6 & 74560.24 & 74218.26 & 73776.17 & \textbf{3.90} \\ \hline
W7 & 252448.52 & 251045.03 & 249459.89 & 4.18 \\ \hline
W8 & 137202.14 & 135952.94 & 133480.08 & 16.17 \\ \hline
W9 & 28351.01 & 28194.96 &28009.15 & \textbf{3.99} \\ \hline
W10 & 92376.49 & 91570.35 & 90172.90 & 17.02 \\ \hline
\end{tabular}
\vspace{0.5cm}

\begin{tabular}{ |l|l|l|l|l|}
\hline
Graph & $\Delta_1^+$ (E) Best & $\Delta_1^+$ (E) Avg & $\Delta_1^+$ (E) Least & $\Delta_1^+$ (E) Time \\ \hline
W1 & \textbf{3622.58}  & \textbf{3580.53} & 3548.82 & 1.41 \\ \hline
W2 & \textbf{36530.25} & \textbf{36191.16} & \textbf{35928.56} &  1.67 \\ \hline
W3 &1603390.76 & 1600505.43& 1596558.94 &185.03\\ \hline
W4 & 320347.01 & 319612.93&318849.26 &195.66\\ \hline
W5 & \textbf{5104.45} & \textbf{5081.95} & \textbf{5063.64} &15.31\\ \hline
W6 & \textbf{75499.50} & \textbf{75175.73} & \textbf{74833.80} &15.70\\ \hline
W7 & \textbf{255793.23} & \textbf{254569.97} & \textbf{253091.91} &15.71\\ \hline
W8 & 137569.32 & \textbf{136896.1} & \textbf{136094.60} &23.83\\ \hline
W9 & 28545.45 & \textbf{28369.43} & \textbf{28141.76} &9.24\\ \hline
W10 & 93021.06 & \textbf{92489.04} & \textbf{91626.99} &25.37\\ \hline
\end{tabular}
\caption{\ref{alg:signlessMBO} cut approximations using $\Delta_1^+$ on randomly weighted graphs, time in seconds.}
\end{table}

\vspace{0.5cm}

\begin{table}
\centering
\begin{tabular}{ |l|l|l|l|l|}
\hline
Graph & $\Delta_s^+$ (S) Best & $\Delta_s^+$ (S) Avg & $\Delta_s^+$ (S) Least & $\Delta_s^+$ (S) Time \\ \hline
W1  &  3601.29  & 3569.23 & 3545.85   & \textbf{0.33} \\ \hline
W2  & 36192.09  & 36059.80  & 35867.83  & 0.49  \\ \hline
W3  & \textbf{1622372.91} & 1620484   & 1618809.76 & 8.40   \\ \hline
W4  & \textbf{323933.40}   & \textbf{323642.4}  & 323114.45  & 7.65  \\ \hline
W5  & 5068.19    & 5041.94   & 5015.16    & 4.50   \\ \hline
W6  & 74844.37   & 74505.45  & 73963.79   & 4.67  \\ \hline
W7  & 253043.96  & 251668.30  & 250600.35  & \textbf{4.12}  \\ \hline
W8  & 137195.52  & 136360.17 & 134856.06  & \textbf{15.38} \\ \hline
W9  & 28389.38   & 28227.09  & 28067.66   & 4.12  \\ \hline
W10 & 92439.42   & 91952.98  & 90488.33   & \textbf{15.33} \\ \hline
\end{tabular}

\vspace{0.5cm}

\begin{tabular}{ |l|l|l|l|l|}
\hline
Graph & $\Delta_s^+$ (E) Best & $\Delta_s^+$ (E) Avg & $\Delta_s^+$ (E) Least & $\Delta_s^+$ (E) Time \\ \hline
W1  &  3614.37   &  3577.56  &  3542.19 & 1.40 \\ \hline
W2  & 36321.80  & 36150.05  & 35910.90 & 1.53  \\ \hline
W3  & 1604257.12 & 1600145.68 & 1597577.4 & 187.88 \\ \hline
W4  & 320691.88  & 319596.27  & 318900.13 & 199.01 \\ \hline
W5  & 5096.55    & 5072.36    & 5041.89   & 15.9   \\ \hline
W6  & 75456.87   & 75089.73   & 74745.17  & 18.09  \\ \hline
W7  & 255316.85  & 253821.64  & 252527.13 & 15.48  \\ \hline
W8  & 137282.02  & 136475.24  & 134333.1  & 24.51  \\ \hline
W9  & 28445.94   & 28258.64   & 28101.22  & 9.18   \\ \hline
W10 & 92731.62   & 92093.05   & 90448.61  & 24.36  \\ \hline
\end{tabular}
\caption{\ref{alg:signlessMBO} cut approximations using $\Delta_s^+$ on randomly weighted graphs, time in seconds.}
\end{table}

\vspace{0.5cm}

\begin{table}
\centering
\begin{tabular}{ |l|l|l|l|l|}
\hline
Graph & $\Delta_0^+$ (S) Best & $\Delta_0^+$ (S) Avg & $\Delta_0^+$ (S) Least & $\Delta_0^+$ (S) Time \\ \hline
W1  &  3413.96  &  3345.32  & 3276.63  &  0.61 \\ \hline
W2  & 34784.30  & 34304.33  & 33627.16  & 0.51  \\ \hline
W3  & 1602346.52 & 1600022.33 & 1595791.12 & \textbf{6.97}   \\ \hline
W4  & 320251.52  & 319940.38  & 319663.40  & \textbf{6.25}   \\ \hline
W5  & 4793.44   & 4761.72    & 4715.51     & 18.66  \\ \hline
W6  & 71219.49   & 70427.83   & 69643.31  & 18.93 \\ \hline
W7  & 239991.72  & 237647.45  & 235617.15  & 19.17  \\ \hline
W8  & 134097.55  & 131088.97  & 126123.70  & 272.56 \\ \hline
W9  & 27528.99   & 26554.77   & 25501.34   & 69.63  \\ \hline
W10 & 90271.70   & 88031.84   & 83130.60   & 264.89 \\ \hline
\end{tabular}

\vspace{0.5cm}

\begin{tabular}{ |l|l|l|l|l|}
\hline
Graph & $\Delta_0^+$ (E) Best & $\Delta_0^+$ (E) Avg & $\Delta_0^+$ (E) Least & $\Delta_0^+$ (E) Time \\ \hline
W1  &  3524.24  &  3456.55  & 3406.93  &  1.03  \\ \hline
W2  & 35664.18  & 35040.71 & 34383.57 &  1.03 \\ \hline
W3  & 1605419.97 & 1602251.82 & 1597064.59 & 203.27  \\ \hline
W4  & 320321.63  & 319809.73  & 319237.08  & 192.51  \\ \hline
W5  & 5017.66    & 4983.90    & 4954.63   & 7.76 \\ \hline
W6  & 74195.87   & 73688.97   & 73231.67  & 7.33  \\ \hline
W7  & 251330.73  & 249754.88  & 248091.06 & 7.51  \\ \hline
W8  &     -     &     -      &    -      & -  \\ \hline
W9  &     -      &    -      &    -      &  - \\ \hline
W10 &     -     &     -      &    -     &  - \\ \hline
\end{tabular}
\caption{\ref{alg:signlessMBO} cut approximations using $\Delta_0^+$ on randomly weighted graphs, time in seconds.}
\end{table}

\vspace{0.5cm}

\begin{table}
\centering
\begin{tabular}{ |l|l|l|l|l|}
\hline
Graph & GW Best & GW Avg & GW Least & GW Time\\ \hline
W1 & 3585.17 & 3535.63 & 3494.26 & 5.74 \\ \hline
W2 & 36101.30 & 35698.47 & 35151.60 & 6.07 \\ \hline
W3 & 1620705.80 & 1618813.52 & 1616502.33 & 43.58 \\ \hline
W4 & 323573.40 & 323275.84 & 322795.83 & 44.09 \\ \hline
W5 & 5038.00 & 5000.74 & 4953.71 & 265.27 \\ \hline
W6 & 74372.75 & 73852.36 & 73293.27 & 241.33 \\ \hline
W7 & 251802.56 & 250316.08 & 248098.85 & 263.44 \\ \hline
W8 & \textbf{138159.14} & 135899.20 & 129576.95 & 2629.60 \\ \hline
W9 & \textbf{28705.35} & 28169.25 & 26422.54 & 689.16 \\ \hline
W10 & \textbf{93547.26} & 91571.68 & 87487.99 & 2646.94 \\ \hline
\end{tabular}
\caption{\ref{alg:GW} cut approximations on randomly weighted graphs, time in seconds.}\label{table:weighted}
\end{table}

\subsection{Large graphs}\label{sec:large}

Since the Euler method has a time complexity of $\mathcal{O}(|E|)$, in this section we show that \ref{alg:signlessMBO} using the Euler method can provide Max-Cut approximations in a respectable time on large sparse datasets. The graphs Amazon0302 and Amazon0601 are networks in which the nodes represent products and an edge exists between two nodes if the corresponding products are frequently co-purchased; both of these networks were constructed in 2003. GNutella31 depicts a peer to peer file sharing network in 2002. PA RoadNet is a road network of Pennsylvania with intersections and endpoints acting as nodes and roads connecting them acting as edges. Email-Enron is a network where each edge represents an email being sent between two people. BerkStan-Web is a network of inter-domain and intra-domain hyperlinks between pages on the domains berkeley.edu and stanford.edu in 2002. Stanford is a network of hyperlinks between pages on the domain stanford.edu in 2002. All of these datasets were obtained from the website \cite{snapnets}. The graph WWW1999 is a model of the Internet in 1999 with edges depicting hyperlinks between websites, obtained from \cite{albert1999}. Table~\ref{table:TabLarge} displays the properties of these graphs. Table~\ref{table:EulerLarge} displays the results we obtained on these graphs choosing $\Delta_1^+$ as our operator, the Euler method as our signless diffusion solver, and $\tau = 10$. For these large graphs, we are unable to obtain results for comparison using \ref{alg:GW}, because \ref{alg:GW} requires too much memory for it to run on the same computer setup.

\begin{table}
\centering
\begin{tabular}{ |l|l|l|l|l|}
\hline
Graph & $|V|$ & $|E|$ & $d_{-}$ & $d_{+}$ \\ \hline
Amazon0302 & 262111 & 899792 & 1 & 420  \\ \hline
Amazon0601 & 403394  & 2443408  & 1 & 2752  \\ \hline
GNutella31 & 62586 & 147892 & 1 & 95  \\ \hline
PA RoadNet & 1088092 & 1541898 & 1 & 9 \\ \hline
Email-Enron & 36692 & 183831 & 1 & 1383  \\ \hline
BerkStan-Web & 685230 & 6649470 & 1 & 84290 \\ \hline
Stanford & 281904 & 1992636 & 1 & 38625  \\ \hline
WWW1999 & 325729 & 1090108 & 1 & 10721 \\ \hline
\end{tabular}
\caption{Properties of our large datasets we are testing on.}\label{table:TabLarge}
\end{table}

\begin{table}
\centering
\begin{tabular}{ |l|l|l|l|l|}
\hline
Graph & $\Delta_1^+$ (E) Best & $\Delta_1^+$ (E) Avg & $\Delta_1^+$ (E) Min & $\Delta_1^+$ (E) Time\\ \hline
Amazon0302 & 618942 & 618512.18 & 618030 & 0.49 \\ \hline
Amazon0601 & 1580070 & 1576960.80 & 1571089 & 1.90 \\ \hline
GNutella31 & 116552 & 116213.74 & 115916 & 0.06 \\ \hline
PA RoadNet & 1380131 & 1379797.90 & 1379416 & 0.64 \\ \hline
Email-Enron & 112665 & 111680.24 & 110279 & 0.02 \\ \hline
BerkStan-Web & 5335813 & 5319662.06 & 5281630 & 0.83 \\ \hline
Stanford & 1585802 & 1580445.14 & 1570469 & 0.47 \\ \hline
WWW1999 & 813000 & 809329.52 & 806130 & 0.21 \\ \hline
\end{tabular}

\caption{Results of \ref{alg:signlessMBO} using $\Delta_1^+$ and the Euler method on large datasets, time in hours.}\label{table:EulerLarge}
\end{table}


\section{Parameter choices}\label{sec:parameters}

\subsection{Variable $K$}

As stated in Section \ref{sec:spectral}, the computational advantage of \ref{alg:signlessMBO} using the spectral method is that not all the eigenpairs of $\Delta^+$ need to be used. In practice, if $K$ is large enough, the cut sizes obtained by \ref{alg:signlessMBO} using the spectral method does not improve significantly when $K$ is increased further. The plots in Figure~\ref{fig:KComp} highlight this. For these three tests we fixed the initial conditions, the choice of operator $\Delta_1^+$, and $\tau = 20$ for each respective graph. For Figure~\ref{fig:K1} we plot the best, average, and least cuts for each choice of $K$. For Figure~\ref{fig:K2} and Figure~\ref{fig:K3} we plot the mean of the best, average, and least cuts over all 100 graphs for each choice of $K$. The error bars indicate the corrected sample standard deviation of the best, average, and least cuts. We ran \ref{alg:signlessMBO} using the spectral method on the AS4 graph, increasing the value for $K$ in increments of 5 from 5 until 100. The plot in Figure~\ref{fig:K1} shows that at $K = 40$ the best, average, and least cut size changes very little for increasing $K$. For Figure~\ref{fig:K2} we ran \ref{alg:signlessMBO} on the 100 realisations of $G(5000,0.001)$ from Section~\ref{sec:Rand}, increasing $K$ in increments of 10 from 10 until 200. For Figure~\ref{fig:K3} we ran \ref{alg:signlessMBO} on the 100 realisations of $R(2500,2,0.009,0.8)$, increasing $K$ in increments of 5 from 5 until 100. The plots in Figure~\ref{fig:K2} and Figure~\ref{fig:K3} show that for our choices of Erd\"os-R\'enyi and random modular graphs increasing $K$ increases the cut sizes. We also note that the best, average, and minimum cut sizes plateau.

\begin{figure}
\centering
\begin{subfigure}{.75\textwidth}
  \centering
  \includegraphics[width=9cm]{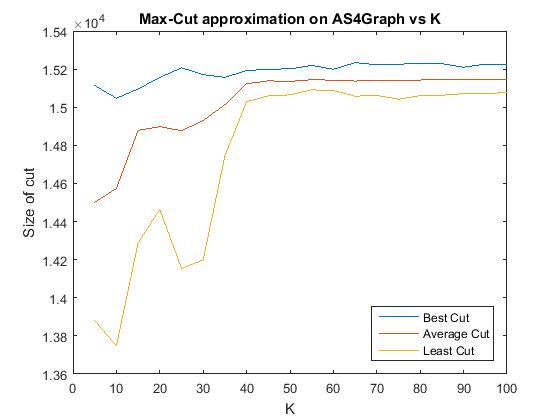}
  \caption{AS4: Cut size as function of $K$, $\tau = 20$.}\label{fig:K1}
\end{subfigure}%

\begin{subfigure}{.75\textwidth}
  \centering
  \includegraphics[width=9cm]{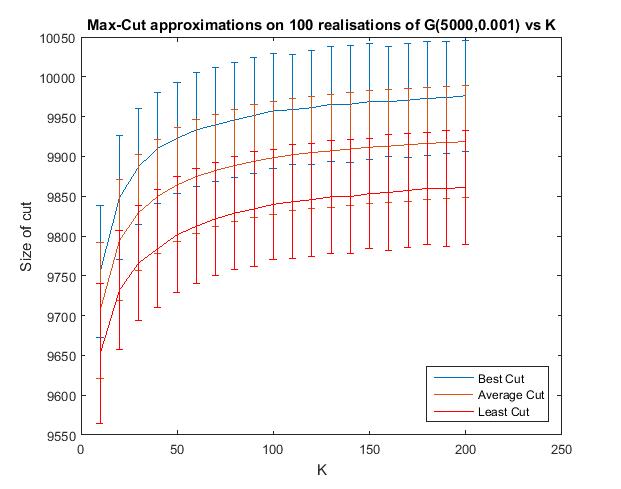}
  \caption{100 realisations of $G(5000,0.001)$: Cut size as function of $K$, $\tau = 20$.}\label{fig:K2}
\end{subfigure}%

\begin{subfigure}{.75\textwidth}
  \centering
  \includegraphics[width=9cm]{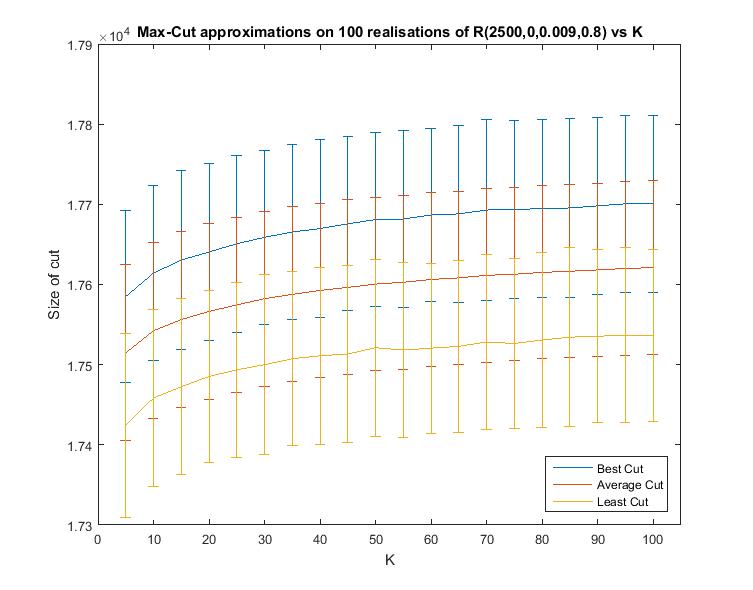}
  \caption{100 realisations of $R(2500,2,0.009,0.8)$: Cut size as function of $K$, $\tau = 20$.}\label{fig:K3}
\end{subfigure}
\caption{Cut size as function of $K$ for three graphs (best viewed in colour).}\label{fig:KComp}
\end{figure}

For large graphs, however, finding the value of $K$ beyond which the produced cut sizes plateau is problematic. We ran \ref{alg:signlessMBO} using the spectral method with $\Delta_1^+$ on the Amazon0302 graph, increasing $K$ in increments of 100 starting from 100 to 2600. As shown in Figure~\ref{fig:KAmazon} the best, average, and least outcomes of \ref{alg:signlessMBO} are still increasing at the end of the range of $K$ values we plotted. For $K = 200$ and $K = 2600$ the run time of \ref{alg:signlessMBO} was 12 minutes and 26 hours, respectively; this increase in computation time resulted in a 3\% increase in cut values. Comparing the cut size obtained for $K = 2600$ with the cut sizes obtained on Amazon0302 in Table~\ref{table:EulerLarge} we see that using the Euler method as the signless diffusion solver is more accurate and significantly faster.

\begin{figure}
\centering
  \includegraphics[width=10cm]{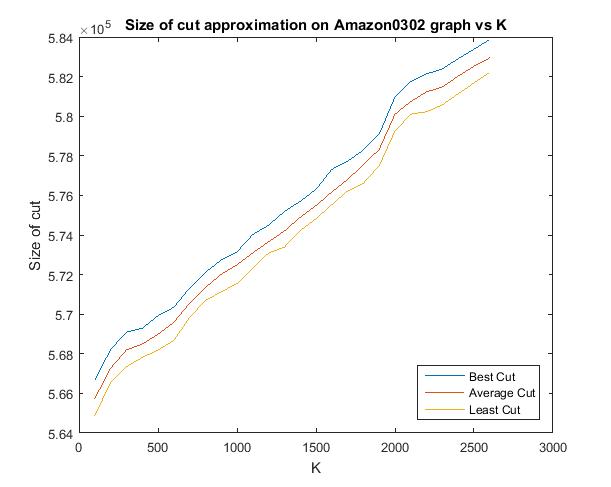}
\caption{Comparison of cut size approximation vs $K$ on Amazon0302 graph.}\label{fig:KAmazon}
\end{figure}

\subsection{Variable $\tau$}

Other than the pinning condition stated in Section~\ref{sec:pinning condition}, currently we have very little information on which to base our choice of $\tau$. In this section we compare the cut sizes obtained by \ref{alg:signlessMBO} against the variable $\tau$. We choose $\Delta_1^+$ as the signless Laplacian operator and the spectral method as the signless diffusion solver. Figure~\ref{fig:taucomp} displays the obtained cut sizes from \ref{alg:signlessMBO} on three (sets of) graphs and compares against $\tau$. For Figure~\ref{fig:tau1} we plot the best, average, and least cuts for each choice of $\tau$. For Figure~\ref{fig:tau2} and Figure~\ref{fig:tau3} we plot the mean of the best, average, and least cuts over all 100 graphs for each choice of $\tau$. The error bars indicate the corrected sample standard deviation of the best, average, and least cuts. We ran \ref{alg:signlessMBO} using the spectral method on the AS4 graph, increasing the value for $\tau$ in increments of 5 starting from 5 until 500. In Figure~\ref{fig:tau1} we see in this experiment that $5 \leq \tau \leq 40$ produces the best results with respect to our cut sizes. We also see that for $330 \leq \tau \leq 480$ the best, average, and least cuts are almost identical. For Figure~\ref{fig:tau2} we ran \ref{alg:signlessMBO} on the 100 realisations of $G(5000,0.001)$ from Section~\ref{sec:Rand}, increasing $\tau$ in increments of 5 starting from 5 until 125. For Figure~\ref{fig:tau3} we ran \ref{alg:signlessMBO} on the 100 realisations of $R(4000,20,0.01,0.7)$, increasing $\tau$ in increments of 5 starting from 5 until 100. In Figure~\ref{fig:tau2} and Figure~\ref{fig:tau3} we see the general trend that increasing $\tau$ beyond 20 decreases the mean over the best, average, and least cuts over all 100 realisations of $G(5000,0.001)$.

\begin{figure}
\centering
\begin{subfigure}{.75\textwidth}
  \centering
  \includegraphics[width=9cm]{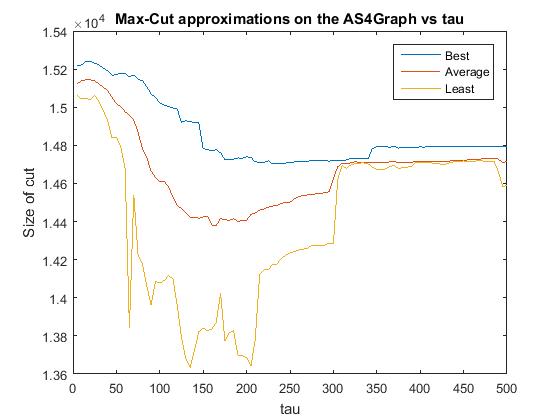}
  \caption{AS4: Cut size as function of $\tau$, $K = 89$.}\label{fig:tau1}
\end{subfigure}%

\begin{subfigure}{.75\textwidth}
  \centering
  \includegraphics[width=9cm]{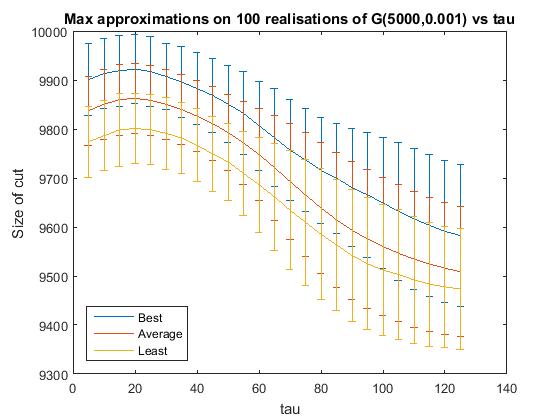}
  \caption{100 realisations of $G(5000,0.001)$: Cut size as function of $\tau$, $K = 49$.}\label{fig:tau2}
\end{subfigure}%

\begin{subfigure}{.75\textwidth}
  \centering
  \includegraphics[width=9cm]{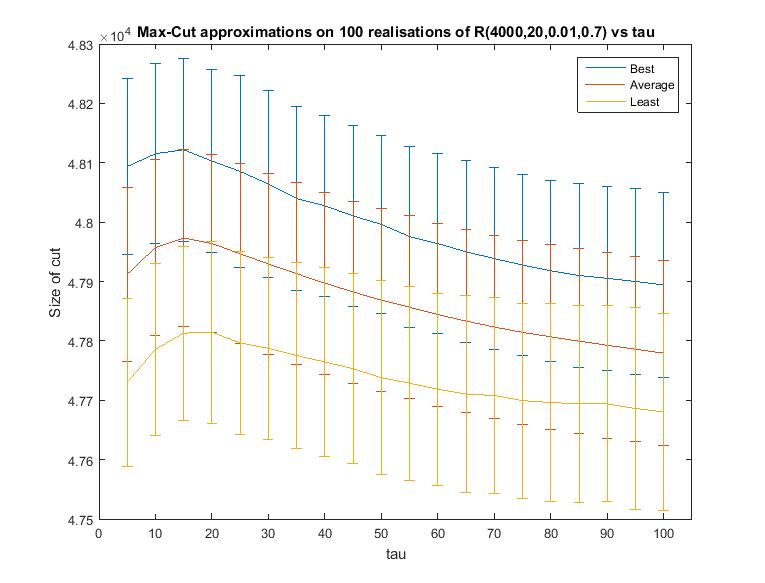}
  \caption{100 realisations of $R(4000,20,0.01,0.7)$: Cut size as function of $\tau$, $K = 40$.}\label{fig:tau3}
\end{subfigure}
\caption{Cut size as function of $\tau$ for three graphs (best viewed in colour).}\label{fig:taucomp}
\end{figure}

\subsection{Implicit Euler scheme}\label{sec:implicit}

On the random graphs we tested on in Section~\ref{sec:Rand} and Section~\ref{sec:mod} our explicit Euler scheme using $\Delta_0^+$ produced non-trivial cut sizes. However, for the scale free graphs in Section~\ref{sec:scale} and Section~\ref{sec:weighted} we did not find a value of $\tau$ or $dt$ such that the cuts induced from \ref{alg:signlessMBO} were non-trivial. In this subsection we show that we can solve the Euler equation implicitly in order to obtain non-trivial cut sizes with the operator $\Delta_0^+$, subject to suitable choices of $dt$ and $\tau$. However, the results are significantly inferior to the operators $\Delta_1^+$ and $\Delta_s^+$ for the implicit Euler scheme. We also compare the \ref{alg:signlessMBO} results obtained using the implicit scheme to the results obtained using the explicit scheme for a set of random graphs.

We run \ref{alg:signlessMBO} using the implicit Euler scheme on the AS4 and AS8 graph from Section~\ref{sec:scale} and the W9 graph from Section~\ref{sec:weighted}. We choose $dt = 0.2$ and $\tau = 20$ when $\Delta_1^+$ or $\Delta_s^+$ is the operator. For $\Delta_0^+$ we set $dt = 0.0005$ and $\tau = 0.05$ for the AS4 graph, and for the AS8 graph and the W9 graph we set $dt = 0.0001$ and $\tau = 0.01$. Table~\ref{tab:impScale} shows that \ref{alg:signlessMBO} using the implicit Euler scheme with $\Delta_0^+$ and our choice of parameters produces cut sizes, however they are significantly smaller in comparison to using $\Delta_1^+$ or $\Delta_s^+$.  

\begin{table}
\centering
\begin{tabular}{|l|l|l|l|}
\hline
Graph & $\Delta_1^+$ Best & $\Delta_s^+$ Best &  $\Delta_0^+$ Best \\ \hline
AS4 & 15276 & \textbf{15279} & 9259 \\ \hline
AS8 & \textbf{23083} & 23033 & 13725 \\ \hline
W9 & \textbf{28553.66} & 28485.28 & 17146.69 \\ \hline
\end{tabular}
\quad
\begin{tabular}{ |l|l|l|l|}
\hline
Graph & $\Delta_1^+$ Avg & $\Delta_s^+$ Avg &  $\Delta_0^+$ Avg \\ \hline
AS4 & \textbf{15196.52} & 15175.52 & 9124.68 \\ \hline
AS8 & \textbf{22934.30} & 22844.16 & 13585.56 \\ \hline
W9 & \textbf{28360.46} & 28294.40 & 16847.92 \\ \hline
\end{tabular}

\vspace{0.5cm}

\begin{tabular}{ |l|l|l|l|}
\hline
Graph & $\Delta_1^+$ Least & $\Delta_s^+$ Least &  $\Delta_0^+$ Least \\ \hline
AS4 & \textbf{15124} & 15056 & 8964 \\ \hline
AS8 & \textbf{22521} & 22454 & 13477 \\ \hline
W9 & \textbf{28103.28} & 28075.62 & 16521.43 \\ \hline
\end{tabular}
\quad
\begin{tabular}{ |l|l|l|l|}
\hline
Graph & $\Delta_1^+$ Time & $\Delta_s^+$ Time &  $\Delta_0^+$ Time \\ \hline
AS4 & 47.83 & 50.94 & \textbf{7.47} \\ \hline
AS8 & 105.22 & 114.74 & \textbf{11.48} \\ \hline
W9 & 38.61 & 42.57 & \textbf{5.26} \\ \hline
\end{tabular}
\caption{Cut sizes obtained by \ref{alg:signlessMBO} using the implicit Euler scheme on scale free graphs, time in seconds.}\label{tab:impScale}
\end{table}

We run \ref{alg:signlessMBO} on the 100 realisations of $G(1000,0.01)$ and $R(4000,20,0.01,0.7)$ in Section~\ref{sec:Rand} and Section~\ref{sec:mod} respectively, using the implicit and explicit Euler method for each operator $\Delta^+ \in \{\Delta_0^+,\Delta_1^+,\Delta_s^+\}$. We choose the same values of $\tau$ and $dt$ as chosen in Section~\ref{sec:Rand} and Section~\ref{sec:mod}, fixing the initial conditions for both methods. Figure~\ref{fig:G1000imp} and Figure~\ref{fig:mod4000imp} show that the average obtained cut sizes using the implicit Euler method are slightly better than the average obtained cut sizes obtained using the explicit method. However, Table~\ref{tab:impexpTime} shows that \ref{alg:signlessMBO} using the explicit Euler method produces cut sizes in less time than using the implicit Euler method on these sets of random graphs. This is why we choose the explicit method for the Euler method in Section~\ref{sec:results}.

\begin{figure}
\centering
  \includegraphics[width=12cm]{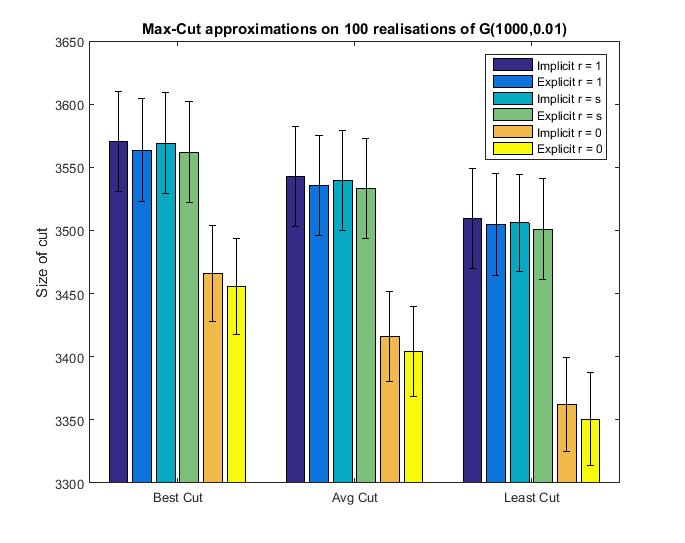}
\caption{Bar chart of Max-Cut approximations on 100 realisations of $G(1000,0.01)$ using the implicit Euler method and the explicit Euler method.}\label{fig:G1000imp}
\end{figure}

\begin{figure}
\centering
  \includegraphics[width=12cm]{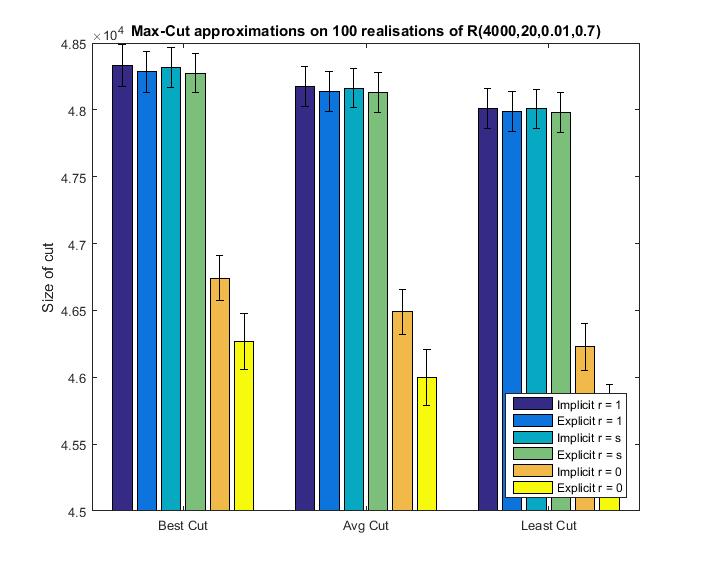}
  \caption{Bar chart of Max-Cut approximations on 100 realisations of $R(4000,20,0.01,0.7)$.}\label{fig:mod4000imp}
\end{figure}

\begin{table}
\begin{tabular}{|l|l|l|l|l|l|l|}
\hline
Graph & $\Delta_1^+$ (I) & $\Delta_1^+$ (E) & $\Delta_s^+$ (I) & $\Delta_s^+$ (E) & $\Delta_0^+$ (I) & $\Delta_0^+$ (E) \\ \hline
$G(1000,0.01)$ & 3.36 & 1.82 & 3.28 & 1.79 & 2.20 & \textbf{1.19} \\ \hline
$R(4000,20,0.01,0.7)$ & 62.97 & 44.23 & 62.16 & 44.18 & 41.53 & \textbf{24.01} \\ \hline
\end{tabular}
\caption{Average \ref{alg:signlessMBO} run-times for each realisation of $G(1000,0.01)$ and $R(4000,20,0.01,0.7)$, time in seconds. (I) indicates the implicit Euler method and (E) indicates the explicit Euler method.}\label{tab:impexpTime}
\end{table}

\section{Conclusions}\label{sec:conclusions}

We have proven that the signless graph Ginzburg-Landau functional $f_{\varepsilon}^+$ $\Gamma$-converges to a Max-Cut objective functional as $\e\downarrow 0$ and thus minimizers of $f_{\varepsilon}^+$ can be used to approximate maximal cuts of a graph. We use an adaptation of the graph MBO scheme involving signless graph Laplacians to approximately minimize $f_\e^+$. We solve the signless diffusion step of our graph MBO scheme using a spectral truncation method and an Euler method.

We tested the resulting \ref{alg:signlessMBO} algorithm on various graphs using both these signless diffusion solvers, and compared the results and run times with those obtained using the \ref{alg:GW} algorithm. In our tests on realizations of random Erd\"os-R\'enyi graphs and on realizations of random modular graphs our \ref{alg:signlessMBO} algorithm using the spectral method outperforms \ref{alg:GW} with reduced run times. On our examples of scale free graphs \ref{alg:GW} usually gives the best maximum cut approximations, but requires run times that are two orders of magnitude longer than those of \ref{alg:signlessMBO}, which obtains cut sizes within about 2\% of those obtained by \ref{alg:GW}. Similar conclusions follow from our tests on weighted graphs, that used randomly generated Erd\"os-R\'enyi graphs and modular graphs, and some scale free graphs, all with random edge weights. We have also shown that our algorithm using the Euler method can be used on large sparse datasets, with reasonable computation times.

In our tests (and for our parameter choices) we see that \ref{alg:signlessMBO} using both $\Delta_1^+$ and $\Delta_s^+$ produces larger Max-Cut approximations than $\Delta_0^+$ for all of the graphs that we tested on.

There are still many open questions related to the \ref{alg:signlessMBO} algorithm, for example questions related to a priori parameter choices (such as $\tau$ and $K$), and performance guarantees. These can be the subject of future research.

\section*{Acknowledgements}
We would like to thank the EPSRC for supporting this work through the DTP grant EP/M50810X/1.
We would also like to thank Matthias Kurzke and Braxton Osting for helpful discussions.
This project has received funding from the European Union's Horizon 2020 research and innovation programme under the Marie Sk\l{}odowska--Curie grant agreement No 777826.


\bibliographystyle{ieeetr}
\bibliography{BibFile}

\vspace{0.5cm}

\noindent
Email address: Blaine.Keetch@nottingham.ac.uk\\
Email address: Y.Vangennip@nottingham.ac.uk

\end{document}